\documentclass[letterpaper, 10pt, conference]{ieeeconf}

\usepackage{cite}
\usepackage{amsmath,amssymb,amsfonts}
\usepackage[scr = dutchcal]{mathalpha}
\usepackage{graphicx}
\usepackage{textcomp}
\usepackage{stmaryrd}
\usepackage{tabularx}
\usepackage{multicol}
\usepackage{multirow}
\usepackage{float}
\usepackage{mathtools}
\usepackage{comment}

\usepackage{color}

\newcommand{\mbf}[1]{\mathbf{ #1}}
\newcommand{\tnf}[1]{\textnormal{#1}}
\newcommand{\tbf}[1]{\textbf{#1}}
\newcommand{\mbs}[1]{\boldsymbol{#1}}

\newcommand{\mcl}[1]{\mathcal{#1}}
\newcommand{\mscr}[1]{\mathscr{#1}}

\newcommand{\R}{\mathbb{R}}
\newcommand{\N}{\mathbb{N}}

\newcommand{\norm}[1]{\left\lVert{#1}\right\rVert}
\newcommand{\ip}[2]{\left\langle #1, #2 \right\rangle}

\newcommand{\bmat}[1]{\begin{bmatrix}#1\end{bmatrix}}

\newcommand{\smallbmat}[1]{\left[\scriptsize\begin{smallmatrix}
		#1\end{smallmatrix} \right]}

\newcommand{\lbmat}[1]{\left[\!\!\begin{array}{l}#1\end{array}\!\!\right]}

\newcommand{\slbmat}[1]{\small\left[\!\!\begin{array}{l}#1\end{array}\!\!\right]}
\newcommand{\sclbmat}[1]{\scriptsize\left[\!\!\begin{array}{l}#1\end{array}\!\!\right]}
\newcommand{\tlbmat}[1]{\tiny\left[\!\!\!\!\begin{array}{l}#1\end{array}\!\!\!\!\right]}
\newcommand{\srbmat}[1]{\small\left[\!\!\begin{array}{r}#1\end{array}\!\!\right]}

\newcommand{\scbrray}[2]{{\scriptsize\left[\!\!\begin{array}{#1}#2\end{array}\!\!\right]}}

\newtheorem{thm}{Theorem}
\newtheorem{defn}[thm]{Definition}
\newtheorem{lem}[thm]{Lemma}
\newtheorem{prop}[thm]{Proposition}
\newtheorem{cor}[thm]{Corollary}

\newtheorem{example}[thm]{Example}

\floatstyle{ruled}
\newfloat{block}{thbp}{lop}
\floatname{block}{Block}

\let\bl\bigl
\let\bbl\Bigl
\let\bbbl\biggl

\let\br\bigr
\let\bbr\Bigr
\let\bbbr\biggr

\usepackage{upgreek}

\newcommand{\PIset}[0]{\mathbf{\Pi}}

\newcommand\blfootnote[1]{%
	\begingroup
	\renewcommand\thefootnote{}\footnote{#1}%
	\addtocounter{footnote}{-1}%
	\endgroup
}

\newcommand\Resize[2]{\resizebox{#1}{!}{\mbox{\ensuremath{\displaystyle #2}}}}

\title{\LARGE \bf
	Representation and Stability Analysis of 1D PDEs with Periodic Boundary Conditions
}

\author{Declan Jagt, Sergei Chernyshenko, Matthew Peet %
}

\begin{document}

	\maketitle
	\pagestyle{plain}

\begin{abstract}
	PDEs with periodic boundary conditions are frequently used to model processes in large spatial environments, assuming solutions to extend periodically beyond some bounded interval. However, solutions to these PDEs often do not converge to a unique equilibrium, but instead converge to non-stationary trajectories existing in the nullspace of the spatial differential operator (e.g. $\frac{\partial^2}{\partial x^2}$). To analyse this convergence behaviour, in this paper, it is shown how such trajectories can be modeled for a broad class of linear, 2nd order, 1D PDEs with periodic as well as more general boundary conditions, using the Partial Integral Equation (PIE) representation. In particular, it is first shown how any PDE state satisfying these boundary conditions can be uniquely expressed in terms of two components, existing in the image and the nullspace of the differential operator $\frac{\partial^2}{\partial x^2}$, respectively. An equivalent representation of linear PDEs is then derived as a PIE, explicitly defining the dynamics of both state components. Finally, a notion of exponential stability is defined that requires only one of the state components to converge to zero, and it is shown how this stability notion can be tested by solving a linear operator inequality. The proposed methodology is applied to two examples, demonstrating that exponential stability can be verified with tight bounds on the rate of decay.%
\end{abstract}

\blfootnote{\vspace*{-0.00cm}%
	\tbf{Acknowledgement:} This work was supported by National Science Foundation grant EPCN-2337751. \vspace*{-0.25cm}}


\section{INTRODUCTION}

Partial Differential Equations (PDEs) are used to model a variety of processes, including physical, biological, chemical, etc. 
For example, the evolution of the temperature distribution $\mbf{u}(t)$ in a rod can be modeled with the heat equation,
\begin{equation}\label{eq:heat_eq_intro}
	\mbf{u}_{t}(t,x)=\mbf{u}_{xx}(t,x),\enspace \mbf{u}(0,x)=\mbf{u}_{0}(x),\quad x\in(0,1).
\end{equation}
Of course, in order to obtain a unique solution, any PDE on a finite domain must be complemented by a set of boundary conditions. 
In practice, these boundary conditions are often chosen to represent behaviour of the system at actual boundaries of the physical environment -- e.g. imposing Dirichlet conditions $\mbf{u}(t,0)=\mbf{u}(t,1)=0$ if the temperature at the endpoints of the rod is known to be zero.
However, in many applications, the endpoints of the interval do not correspond to physical boundaries, but instead, the interval is used to represent an (infinitely) repeating segment of a larger domain. Assuming solutions in each segment to be identical, so that $\mbf{u}(t,x+n):=\mbf{u}(t,x)$ for $n\in\N$, this can be modeled using periodic boundary conditions, setting e.g. $\mbf{u}(t,0)=\mbf{u}(t,1)$ and $\mbf{u}_{x}(t,0)=\mbf{u}_{x}(t,1)$. 

Although periodic boundary conditions offer a crucial tool for modeling processes in larger spatial domains, analysing stability of equilibrium solutions for these conditions may be challenging. For example, any uniform distribution $\mbf{u}^*\equiv \beta$ for $\beta\in\R$ will satisfy the periodic conditions and $\mbf{u}^*_{xx}=0$, allowing infinitely many equilibrium states for the heat equation.
Moreover, these equilibria are non-isolated, in that any one equilibrium exists in an arbitrarily small neighborhood of another. As a result, behaviour of solutions is highly dependent on the initial conditions, with $\int_{0}^{1}\mbf{u}_{0}(x)dx=\beta$ implying $\lim_{t\to\infty}\mbf{u}(t)\equiv \beta$, complicating stability analysis.

Of course, non-isolated equilibria do not occur exclusively for periodic boundary conditions. Indeed, the uniform solutions $\mbf{u}^*\equiv \beta$ are also equilibria for the heat equation with Neumann conditions $\mbf{u}_{x}(t,0)=\mbf{u}_{x}(t,1)=0$. Moreover, any affine distribution $\mbf{u}^*(x)=\alpha x+\beta$ will be an equilibrium for the heat equation when using $\mbf{u}(t,1)=\mbf{u}(t,0)+\mbf{u}_{x}(t,0)$ and $\mbf{u}_{x}(t,0)=\mbf{u}_{x}(t,1)$.
The problem with all of these boundary conditions, of course, is that for $\mbf{u}$ in the space constrained by these conditions -- i.e. the PDE domain -- $\mbf{u}_{xx}=0$ does not imply $\mbf{u}=0$.
More precisely, the nullspace of the differential operator $\mscr{D}:=\frac{\partial^2}{\partial x^2}$ defining the heat equation is not trivial on the PDE domain.
By contrast, for Dirichlet conditions, this operator $\frac{\partial^2}{\partial x^2}$ is actually invertible on the PDE domain, admitting only a single equilibrium $\mbf{u}^*\equiv 0$.

Fortunately, most practical PDE models involve e.g. reaction or nonlinear terms that prohibit non-isolated equilibria, allowing stability properties to be accurately tested even for Neumann or periodic boundary conditions. 
For example, for reaction-diffusion equations with Neumann (and Dirichlet) conditions, stability conditions for observer-based controllers have been presented in~\cite{katz2020controlLMI,katz2022controlLMI} and~\cite{lhachemi2023SaturatedStabilization,lhachemi2023BoundaryStabilization}, and stabilization may also be performed using backstepping~\cite{deutscher2018backstepping,mathiyalagan2024backstepping}.
Similarly, stability for the Navier-Stokes equations with arbitrary boundary conditions can be tested using a Sum-Of-Squares (SOS) approach as in~\cite{chernyshenko2014SOS_NS,fuentes2022SOS_NS}, and for periodic boundary conditions, an observer can be designed as in~\cite{zayats2021ObserverNS,zhuk2023ObserverNS}.
More comprehensive frameworks for stability analysis of linear and nonlinear PDEs with general boundary conditions have been developed as well, using e.g. SOS~\cite{valmorbida2015StabilityPDE,ahmadi2016AnalysisPDEs}, or occupation measures and the moment-SOS hierarchy~\cite{magron2020MomentsControl,korda2022Moments,henrion2023MomentsPDEs}.

Using these various results, stability of equilibrium solutions can be tested for a broad class of PDEs. However, although analysing stability of a single equilibrium suffices for most practical applications, in doing so, certain insight on convergence behaviour of solutions may be lost. For example, for the heat equation with periodic boundary conditions, although the trivial solution $\mbf{u}^*\equiv 0$ is merely Lyapunov stable, each solution does converge asymptotically to an equilibrium $\beta:=\int_{0}^{1}\mbf{u}_{0}(x)dx$. Similarly, adding a reaction term as $\mbf{u}_{t}(t,x)=\mbf{u}_{xx}(t,x)+\mbf{u}(t,x)$, the equilibrium $\mbf{u}^*\equiv 0$ becomes unstable, but solutions do still converge as $\norm{\mbf{u}(t)-\beta(t)}_{L_{2}}\to 0$, where now $\beta(t):=e^{t}\int_{0}^{1}\mbf{u}_{0}(x)dx$.

Of course, this kind of convergence $\norm{\mbf{u}(t)-\beta(t)}_{L_{2}}\to 0$ could be tested by expanding solutions as $\mbf{u}(t,x)=\beta(t)+\sum_{i\in\N}c_{i}(t)\mbs{\phi}_{i}(x)$ for some basis $\{\mbs{\phi}_{i}\}$ -- as done in e.g. ~\cite{lhachemi2023BoundaryStabilization,fuentes2022SOS_NS,magron2020MomentsControl} -- and analysing the behaviour of $\mbf{u}(t)-\beta(t)=\sum_{i\in\N}c_{i}(t)\mbs{\phi}_{i}(x)$.
In this paper, however, we propose a framework that allows such analysis to be performed without the need for a basis expansion, 
by establishing a map $\mcl{T}$ such that $\mbf{u}(t)-\beta(t)=\mcl{T}\mbf{u}_{xx}(t)$, and explicitly modeling the dynamics of $\mcl{T}\mbf{u}_{xx}(t)$.
To achieve this, we will use the framework presented in e.g.~\cite{shivakumar2024GPDE}, in which it was shown that for a broad class of linear boundary conditions, including Dirichlet and Robin conditions, the differential operator $\frac{\partial^2}{\partial x^2}$ admits a unique inverse $\mcl{T}$ on the PDE domain. Moreover, this inverse
is defined by a Partial Integral (PI) operator, 
parameterized by polynomials $\mbs{T}_{1},\mbs{T}_{2}$ as
\begin{align*}
	\ \\[-2.05em]
	(\mcl{T}\mbf{v})(x)=\int_{a}^{x}\mbs{T}_{1}(x,\theta)\mbf{v}(\theta)d\theta +\int_{x}^{b}\mbs{T}_{2}(x,\theta)\mbf{v}(\theta)d\theta. \\[-1.95em]
\end{align*} 
Then, for a broad class of linear PDEs, we can define another PI operator $\mcl{A}$ such that $\mbf{u}$ solves the PDE if and only if $\mbf{v}=\mbf{u}_{xx}$ solves the Partial Integral Equation (PIE) \\[-1.0em]
\begin{equation*}
	\partial_{t} \mcl{T}\mbf{v}(t)=\mcl{A}\mbf{v}(t).
\end{equation*}
\ \\[-1.6em]
In this representation, the \textit{fundamental state} $\mbf{v}(t)\in L_{2}$ does not have to satisfy any boundary conditions or regularity constraints. Exploiting the algebraic properties of PI operators, as well as their polynomial parameterization, this allows a variety of problems such as stability analysis, optimal control, and optimal estimation for linear ODE-PDE systems to be posed as operator inequalities on PI operator variables~\cite{shivakumar2024thesis}, which can be solved using semidefinite programming~\cite{shivakumar2025PIETOOLS2024}.

Of course, constructing a PIE representation for PDEs with periodic boundary conditions is complicated by the fact that the differential operator $\frac{\partial^2}{\partial x^2}$ is not invertible on the resulting PDE domain. Indeed, $\frac{\partial^2}{\partial x^2}$ is not injective on this PDE domain, since its nullspace includes any uniform solution, $\mbf{u}\equiv \beta$.
Moreover, $\frac{\partial^2}{\partial x^2}$ is also not surjective onto $L_{2}$, since $\mbf{u}(t)$ will satisfy $\mbf{u}_{x}(t,0)=\mbf{u}_{x}(t,1)$ only if $\int_{0}^{1}\mbf{u}_{xx}(t,x)dx=0$, imposing constraints directly on the fundamental state $\mbf{v}(t)=\mbf{u}_{xx}(t)$.

Extending the PIE representation to support periodic as well as more general boundary conditions, then, we face several challenges. First, given any set of boundary conditions, how do we define the range of $\frac{\partial^2}{\partial x^2}$ on the resulting PDE domain?
Secondly, how do we account for the potentially nontrivial nullspace of $\frac{\partial^2}{\partial x^2}$ on the PDE domain, identifying operators $\mcl{T}$ and $\mcl{Q}$ such that for any $\mbf{u}(t)$ in this domain, $\frac{\partial^2}{\partial x^2}\mcl{Q}\mbf{u}(t)=0$ and $\mbf{u}(t)=\mcl{T}\mbf{u}_{xx}(t)+\mcl{Q}\mbf{u}(t)$? 
Finally, how do we construct a PIE modeling the dynamics of not only $\mbf{v}(t)=\mbf{u}_{xx}(t)$ but also $\mcl{Q}\mbf{u}(t)$, and how do we test stability in this representation?

In the remainder of this paper, we resolve each of these challenges, extending the PIE representation to support PDEs with periodic as well as more general boundary conditions. To this end, in Sec.~\ref{sec:Tmap} we first show how for a general linear PDE domain $X\subseteq W_{2}^{2,n}$, we can define the image of $\frac{\partial^2}{\partial x^2}$ on $X$ as an $L_{2}$-subspace $Y$. Next, we show how we can define a functional $\mcl{F}$ and PI operators $\mcl{T}_{0},\mcl{T}_{1}$ such that $\mbf{u}=\mcl{T}_{0}\mcl{F}\mbf{u}+\mcl{T}_{1}\mbf{u}_{xx}$ for all $\mbf{u}\in X$ -- where now $\frac{\partial^2}{\partial x^2}(\mcl{T}_{0}\mcl{F}\mbf{u})=0$. 
Using this relation, in Sec.~\ref{sec:PDE2PIE}, we then derive an equivalent PIE representation of a class of linear PDEs, modeling the fundamental state $\mbf{v}(t)=(\mcl{F}\mbf{u}(t),\mbf{u}_{xx}(t))\in\R^{m}\times Y$. 
Although this PIE representation is derived only for linear, 2nd-order PDEs, such a PIE representation can be similarly derived for higher-order PDEs, coupled ODE-PDEs, and nonlinear PDEs as well.
Finally, in Sec.~\ref{sec:LPI}, we show how stability in this PIE representation can be tested as a linear PI operator inequality, allowing convergence of solutions in both the sense $\norm{\mbf{u}(t)}_{L_{2}}\to 0$ and $\norm{\mbf{u}(t)-\mcl{T}_{0}\mcl{F}\mbf{u}(t)}_{L_{2}}\to 0$ to be tested.
We solve this operator inequality for a heat equation and a wave equation in Sec.~\ref{sec:examples}.

\section{Notation}

For a given interval $[a,b]\subset\R$, let $L_{2}^{n}[a,b]$ denote the Hilbert space of $\R^{n}$-valued square-integrable functions on $[a,b]$, where we omit the domain when clear from context. Denote $\partial_{x}^2:=\frac{\partial^2}{\partial x^2}$, and define the Sobolev subspace
\begin{equation*}
	W_{2}^{2,n}[a,b]:=\bl\{\mbf{u}\in L_{2}^{n}[a,b]\mid \mbf{u}_{x},\mbf{u}_{xx}\in L_{2}[a,b]\br\},
\end{equation*}
where $\mbf{u}_{x}=\partial_{x}\mbf{u}$ and $\mbf{u}_{xx}=\partial_{x}^2\mbf{u}$.
Let $\R^{m\times n}[x,\theta]$ denote the space of $m\times n$ matrix-valued polynomials in variables $x,\theta$. For $\mbs{R},\mbs{Q}\in L_{2}^{m\times n}[a,b]$, define the multiplier operator $\tnf{M}_{\mbs{R}}:\R^{n}\to L_{2}^{m}$ and functional $\smallint_{a}^{b}[\mbs{Q}]:L_{2}^{n}\to\R^{m}$ by
$(\tnf{M}_{\mbs{R}}v)(x):=\mbs{R}(x)v$ and $\smallint_{a}^{b}[\mbs{Q}]\mbf{u}:=\int_{a}^{b}\mbs{Q}(x)\mbf{u}(x)dx$.

\section{A Bijection Between $\mbf{u}$ and $\mbf{u}_{xx}$}\label{sec:Tmap}

In this section, we provide the main technical contribution of the paper, proving that for a PDE domain of the form 
\begin{equation}\label{eq:Xset}
	X:=\bbl\{\mbf{u}\in W_{2}^{2,n}\,\bbl|\, E\smallbmat{\mbf{u}(a)\\\mbf{u}(b)\\\mbf{u}_{x}(a)\\\mbf{u}_{x}(b)}+\int_{a}^{b}\!\!\mbs{F}(x)\mbf{u}(x)dx=0\bbr\},
\end{equation}
the map $\partial_{x}^2:X\to L_{2}^{n}$ is invertible only if $\{E,\mbs{F}\}\in\R^{2n\times 4n}\times L_{2}^{2n\times n}$ is such that the matrix $G_{E,\mbs{F}}\in\R^{2n\times 2n}$ (defined in the Subsec.~\ref{sec:Tmap:Yset}) is of full rank. Moreover, if this condition fails, we can still define an $L_{2}$ subspace,
\begin{equation}\label{eq:Yset_K}
	Y:=\{\mbf{v}\in \R^{m}\times L_2^{n} \mid \mcl{K}\mbf{v}=0\},
\end{equation}
for some functional $\mcl{K}:\R^{m}\times L_{2}^{n}\to\R^{m}$,
such that $X$ is isomorphic to $Y$, 
with $\mscr{D}:X\to Y$ defined by a differential operator and the inverse $\mcl{T}=\mscr{D}^{-1}:Y\to X$ defined by an integral operator with semiseparable kernel. 
This result allows us to define an equivalent parameterization of the state of the PDE free of boundary conditions and Sobolev regularity constraints.
We prove the result in three steps:
 
\begin{enumerate}
	\item First, in Subsection~\ref{sec:Tmap:Yset}, we show that for any $\{E,\mbs{F}\}$, we can define $\mbs{K}\in L_{2}^{m\times n}$ such that $\mbf{u}\in X$ only if $\mbf{u}_{xx}\in \hat{Y}:=\{\mbf{v}\in L_{2}^{n}\mid \smallint_{a}^{b}[\mbs{K}]\mbf{v}=0\}$.
	
	\item Next, in Subsection~\ref{sec:Tmap:extraBCs}, we show how we can define a functional $\mcl{F}$ and an integral operator $\mcl{T}_{1}$ such that for all $\mbf{v}\in \hat{Y}$, $\mcl{T}_{1}\mbf{v}\in X$ and $\sclbmat{\mcl{F}\\\partial_{x}^2}\mcl{T}_{1}\mbf{v}=\sclbmat{0\\\mbf{v}}$. 
	
	\item Finally, in Subsection~\ref{sec:Tmap:Tmap}, we define $\mcl{T}_{0}$ such that for all $v_{0}\in\R^{m}$, $\mcl{T}_{0}v_{0}\in X$ and $\sclbmat{\mcl{F}\\\partial_{x}^2}\mcl{T}_{0}v_{0}=\sclbmat{v_{0}\\0}$. Defining $\mscr{D}:=\sclbmat{\mcl{F}\\\partial_{x}^2}$ and $\mcl{T}:=\bmat{\mcl{T}_{0}&\mcl{T}_{1}}$, it follows that $\mbf{u}=\mcl{T}\mscr{D}\mbf{u}$ for all $\mbf{u}\in X$.
\end{enumerate}

\subsection{The Range of $\partial_{x}^2$ on the PDE Domain}\label{sec:Tmap:Yset}

Consider the subspace $X\subseteq W_{2}^{2,n}$ constrained by $2n$ linear boundary conditions as in~\eqref{eq:Xset}. 
In order to define an inverse of the differential operator $\partial_{x}^2:X\to L_{2}^{n}$ on this subspace, we first need to define the range of $\partial_{x}^2$ on $X$, which may be a proper subspace of $L_{2}^{n}$. For example, if $X$ imposes the periodic condition $\mbf{u}_{x}(a)=\mbf{u}_{x}(b)$, then $\mbf{u}\in X$ implies $\int_{a}^{b}\mbf{u}_{xx}(x)dx=0$, thus restricting the range of $\partial_{x}^2$ on $X$. In this subsection we show that, more generally, for any $\{E,\mbs{F}\}$, we can define an associated $\mbs{K}\in L_{2}^{m\times n}$ such that $\mbf{u}\in X$ implies $\int_{a}^{b}\mbs{K}(x)\mbf{u}_{xx}(x)dx=0$.
Naturally, this first requires representing the boundary conditions in terms of the second-order derivative, $\mbf{u}_{xx}$. We define such a representation of the boundary conditions in terms of a matrix $G_{E,\mbs{F}}$ and function $\mbs{H}_{E,\mbs{F}}$ as follows.

\begin{defn}\label{defn:Gmat}
For given $E\in \R^{m\times 4n}$ and $\mbs{F}\in L_{2}^{m\times n}$, we define $G_{E,\mbs{F}}\in \R^{m \times 2n}$ and $\mbs{H}_{E,\mbs{F}}\in L_2^{m\times n}$ by
	\begin{align*}
		&G_{E,\mbs{F}}:=E\scbrray{cc}{I_{n}&0_{n}\\I_{n}&(b-a)I_{n}\\0_{n}&I_{n}\\0_{n}&I_{n}}+\int_{a}^{b}\mbs{F}(x)\bmat{I_{n}&(x-a)I_{n}}dx,\\
		&\mbs{H}_{E,\mbs{F}}(x):=-E{\sclbmat{0_{n}\\(b-x)I_{n}\\0_{n}\\I_{n}}}-\int_{x}^{b}(\theta-x)\mbs{F}(\theta)d\theta,
	\end{align*}
	and we define the subspace $X_{E,\mbs{F}}\subseteq W_{2}^{2,n}$ by
	\begin{align*}
		&X_{E,\mbs{F}}:=	\\[-0.6em]
		&\quad\bbl\{\mbf u\in W_{2}^{2,n}\,\bbl|\,		G_{E,\mbs{F}}
		{\slbmat{\mbf{u}(a)\\\mbf{u}_{x}(a)}}=\int_{a}^{b}\!\!\mbs{H}_{E,\mbs{F}}(x)\mbf{u}_{xx}(x)dx\bbr\}.
	\end{align*}
\end{defn}

\smallskip
The following lemma shows that the boundary conditions defining $X_{E,\mbs{F}}$ are in fact equivalent to those defining $X$.
\begin{lem}\label{lem:BC_expansion}
For $E\in \R^{m\times 4n}$ and $\mbs{F}\in L_{2}^{m\times n}$, let $X$ be as in~\eqref{eq:Xset}, and $X_{E,\mbs{F}}$ as in Defn.~\ref{defn:Gmat}. Then $X_{E,\mbs{F}}=X$.
\end{lem}
\begin{proof}
	The result follows using Taylor's theorem with integral form of the remainder, by which, for any $\mbf{u}\in W_{2}^{2,n}$,
	\begin{equation*}
		\mbf{u}(x)=\mbf{u}(a) +(x-a)\mbf{u}_{x}(a) +\int_{a}^{x}(x-\theta)\mbf{u}_{xx}(\theta)d\theta.
	\end{equation*}
	Using this identity, the values of $\smallbmat{\mbf{u}(b)\\\mbf{u}_{x}(b)}$ and $\int_{a}^{b}\mbs{F}(x)\mbf{u}(x)dx$ in the definition of $X$ can all be expressed in terms of only $\mbf{u}(a)$, $\mbf{u}_{x}(a)$, and $\mbf{u}_{xx}(x)$, yielding the proposed representation of $X_{E,\mbs{F}}$ in terms of $G_{E,\mbs{F}}$ and $\mbs{H}_{E,\mbs{F}}$. A full proof is given in Lem.~\ref{lem:BC_expansion_appx}, in Appx.~\ref{appx:proofs}.
\end{proof}

By Lem.~\ref{lem:BC_expansion}, the boundary conditions on $\mbf{u}\in X$ can be equivalently expressed in terms of some linear combination of the lower-boundary values $(\mbf{u}(a),\mbf{u}_{x}(a))$, and some functional of $\mbf{u}_{xx}$. 
Here, if the matrix $G_{E,\mbs{F}}$ is of full rank -- as is the case for e.g. Dirichlet conditions -- then all boundary conditions will involve the boundary values $(\mbf{u}(a),\mbf{u}_{x}(a))$, not imposing any constraint directly on $\mbf{u}_{xx}$.
As such, the differential operator $\partial_{x}^2:X=X_{E,\mbs{F}}\to L_{2}^{n}$ is surjective in this case, and in fact, we can define an inverse $\mcl{T}:L_{2}^{n}\to X$ as in e.g.~\cite{shivakumar2024GPDE}.
However, if $G_{E,\mbs{F}}$ is not of full rank -- as is the case for e.g. periodic conditions -- certain boundary conditions may impose constraints directly on $\mbf{u}_{xx}$. To isolate these constraints, we partition the matrix $G_{E,\mbs{F}}$ defining $X_{E,\mbs{F}}$ into full-rank and zero-rank parts as follows.

\begin{lem}\label{lem:BC_split}
For any $E\in \R^{2n\times 4n},\mbs{F}\in L_{2}^{2n\times n}$, there exists $0\leq m\leq 2n$ and an invertible $J\in\R^{2n\times 2n}$ such that 
\begin{equation*}
	JG_{E,\mbs{F}}=G_{JE,J\mbs{F}}=G_{\tlbmat{E_1\\E_2},\tlbmat{\mbs{F}_1\\ \mbs{F}_2}}=\bmat{G_{E_1,\mbs{F}_1}\\0_{m\times 2n}},
\end{equation*}
where $G_{E_{1},\mbs{F}_{1}}$ has full row rank and $X=X_{\tlbmat{E_{1}\\E_{2}},\tlbmat{\mbs{F}_{1}\\\mbs{F}_{2}}}$.
\end{lem}
\begin{proof}
	Let $m$ be such that $G_{E,\mbs{F}}\in\R^{2n\times 2n}$ is of rank $2n-m$. Using Gauss-Jordan elimination, we can define an invertible matrix $J\in\R^{2n\times 2n}$ such that $JG_{E,\mbs{F}}$ is in reduced row echelon form, and therefore $JG_{E,\mbs{F}}=\sclbmat{A\\0_{m\times 2n}}$ for some $A\in\R^{2n-m\times 2n}$ of full row rank. Let $\{E_{1},\mbs{F}_{1}\}$ and $\{E_{2},\mbs{F}_{2}\}$ be given by the first $2n-m$ and last $m$ rows of $\{JE,J\mbs{F}\}$, respectively, so that $JE=\sclbmat{E_{1}\\E_{2}}$ and $J\mbs{F}=\sclbmat{\mbs{F}_{1}\\\mbs{F}_{2}}$. By definition of $G_{E,\mbs{F}}$, it follows that $G_{E_{1},\mbs{F}_{1}}=A$ and $G_{E_{2},\mbs{F}_{2}}=0_{m\times 2n}$. Finally, since $J$ is invertible, we have $\mbf{u}\in X$ if and only if
	\begin{equation*}
		JE\smallbmat{\mbf{u}(a)\\\mbf{u}(b)\\\mbf{u}_{x}(a)\\\mbf{u}_{x}(b)}+\int_{a}^{b}J\mbs{F}(x)\mbf{u}(x)dx=0,
	\end{equation*}
	whence $X_{E,\mbs{F}}=X_{JE,J\mbs{F}}=X_{\tlbmat{E_{1}\\E_{2}},\tlbmat{F_{1}\\F_{2}}}$. By Lem.~\ref{lem:BC_expansion}, it follows that $X=X_{\tlbmat{E_{1}\\E_{2}},\tlbmat{F_{1}\\F_{2}}}$.
\end{proof}

By Lem.~\ref{lem:BC_split}, if the boundary conditions defining $X$ are such that $G_{E,\mbs{F}}$ is rank-defficient, then we can partition the space as $X=X_{E_{1},\mbs{F}_{1}}\cap X_{E_{2},\mbs{F}_{2}}$, where now $G_{E_{1},\mbs{F}_{1}}$ is of full rank whereas $G_{E_{2},\mbs{F}_{2}}=0$. 
Here, the boundary conditions defined by $\{E_{2},\mbs{F}_{2}\}$ do not involve the boundary values $(\mbf{u}(a),\mbf{u}_{x}(a))$ at all, but rather, impose a constraint directly on $\mbf{u}_{xx}$. Consequently, the range of $\partial_{x}^2$ on $X_{E_{2},\mbs{F}_{2}}$ will be a proper subspace of $L_{2}^{n}$, yielding the following corollary.

\begin{cor}\label{cor:Yset}
	Let $E_{2}\in\R^{m\times 2n}$ and $\mbs{F}_{2}\in L_{2}^{m\times n}[a,b]$ be such that $G_{E_{2},\mbs{F}_{2}}=0$, and define
	\begin{equation*}
		\hat{Y}:=\bbbl\{\mbf{v}\in L_{2}^{n}\,\bbl|\, \int_{a}^{b}\mbs{H}_{E_{2},\mbs{F}_{2}}(x)\mbf{v}(x)dx=0\bbbr\}.
	\end{equation*}
	Then $\mbf{u}\in X_{E_{2},\mbs{F}_{2}}$ if and only if $\mbf{u}\in W_{2}^{2,n}$ and $\mbf{u}_{xx}\in \hat{Y}$.
\end{cor}
\begin{proof}
	Fix arbitrary $\mbf{u}\in W_{2}^{2,n}$. Since $G_{E_{2},\mbs{F}_{2}}=0$, it follows from Lem.~\ref{lem:BC_expansion} that $\mbf{u}\in X_{E_{2},\mbs{F}_{2}}$ if and only if $\int_{a}^{b}\mbs{H}_{E_{2},\mbs{F}_{2}}(x)\mbf{u}_{xx}(x)dx=0$, and thus $\mbf{u}_{xx}\in \hat{Y}$.
\end{proof}

By Cor.~\ref{cor:Yset}, the range of the differential operator $\partial_{x}^2$ on $X_{E_{2},\mbs{F}_{2}}$ is a subspace $\hat{Y}\subseteq L_{2}^{n}$, defined by some functional constraint. It follows that, unless $G_{E,\mbs{F}}$ is of full rank, the range of $\partial_{x}^2$ on $X=X_{E_{1},\mbs{F}_{1}}\cap X_{E_{2},\mbs{F}_{2}}$ will be a proper subspace $\hat{Y}\subset L_{2}^{n}$ as well. In the following subsection, we show that $\partial_{x}^{2}:X\to \hat{Y}$ is in fact surjective onto this range, defining a right-inverse $\mcl{T}_{1}:\hat{Y}\to X$ to $\partial_{x}^2$ as an integral operator with polynomial semi-separable kernel.

\subsection{A Right-Inverse of $\partial_{x}^2$ on the PDE Domain}\label{sec:Tmap:extraBCs}

Having defined the range $\hat{Y}\subseteq L_{2}^{n}$ of the differential operator $\partial_{x}^2$ on $X$, we now propose an inverse of the differential operator on this range, as an integral operator $\mcl{T}_{1}:\hat{Y}\to X$.
Of course, in general, the (right-)inverse of a differential operator is not unique. 
Indeed, letting
\begin{equation}\label{eq:Tmap_Green}
	(\mcl{T}_{1}\mbf{v})(x):=\alpha(x-a)+\beta +\int_{a}^{b}(x-\theta)\mbf{v}(\theta)d\theta,
\end{equation}
for any $\alpha,\beta\in\R^{n}$, the map $\mcl{T}_{1}:L_{2}^{n}\to W_{2}^{2,n}$ defines a right-inverse to $\partial_{x}^2$.
The challenge, then, is to choose $\alpha$ and $\beta$ in such a manner that $\mcl{T}_{1}\mbf{v}\in X=X_{E_{1},\mbs{F}_{1}}\cap X_{E_{2},\mbs{F}_{2}}$. Here, by Cor.~\ref{cor:Yset}, we already have $\mcl{T}_{1}\mbf{v}\in X_{E_{2},\mbs{F}_{2}}$ if and only if $\mbf{v}\in \hat{Y}$, for any $\alpha,\beta$. 
Now, to find $\alpha,\beta$ such that $\mcl{T}_{1}\mbf{v}\in X_{E_{1},\mbs{F}_{1}}$, note that $\mbf{u}=\mcl{T}_{1}\mbf{v}$ implies $\mbf{u}(a)=\beta$ and $\mbf{u}_{x}(a)=\alpha$, and therefore $\mbf{u}\in X_{E_{1},\mbs{F}_{1}}$ if and only if
\begin{equation*}
	G_{E_{1},\mbs{F}_{1}}{\slbmat{\beta\\\alpha}}
	=G_{E_{1},\mbs{F}_{1}}{\slbmat{\mbf{u}(a)\\\mbf{u}_{x}(a)}}
	=\int_{a}^{b}\!\mbs{H}_{E_{1},\mbs{F}_{1}}(x)\mbf{v}(x)dx.
\end{equation*}
Given $\mbf{v}$, these equations could be readily solved for $\alpha,\beta$, were it not for the fact that $G_{E_{1},\mbs{F}_{1}}\in\R^{2n-m\times 2n}$ has only $2n-m$ rows, whereas we have $2n$ unknowns ($\smallbmat{\beta\\\alpha}\in\R^{2n}$).
Therefore, unless the matrix $G_{E,\mbs{F}}$ is of full rank, this still leaves $m$ degrees of freedom when defining the operator $\mcl{T}_{1}$. Rather than choosing these variables arbitrarily, we propose to define a set of $m$ auxiliary boundary conditions, $\{E_{3},\mbs{F}_{3}\}$, and choose $\alpha,\beta$ such that $\mcl{T}_{1}\mbf{v}\in X_{\tlbmat{E_{1}\\E_{3}},\tlbmat{\mbs{F}_{1}\\\mbs{F}_{3}}}$. Naturally, this requires choosing $\{E_{3},\mbs{F}_{3}\}$ such that $G_{\tlbmat{E_{1}\\ E_{3}},\tlbmat{\mbs{F}_{1}\\\mbs{F}_{3}}}$ is invertible. The following lemma shows how this may be achieved, proving that, in fact, we may always let $E_{3}=0$.

\begin{lem}\label{lem:BC_extra}
	Let $\{E_{1},\mbs{F}_{1}\}\in\R^{2n-m\times 4n}\times L_{2}^{2n-m\times n}$, and let $P:=\smallbmat{P_{11}&P_{12}\\P_{21}&P_{22}}$ 
	for $P_{11},P_{21}\in\R^{n\times 2n-m}$ and $P_{12},P_{22}\in\R^{n\times m}$ be a permutation matrix such that $G_{E_{1},\mbs{F}_{1}}P=\bmat{I_{2n-m}&M}$ for some $M\in\R^{2n-m\times m}$. Define
	\begin{equation*}
		\mbs{F}_{3}(x):=\frac{2(2b+a-3x)}{(b-a)^2}P_{12}^T -\frac{6(b+a-2x)}{(b-a)^3}P_{22}^T.
	\end{equation*}
	Then, the matrix $G_{\tlbmat{E_{1}\\0},\tlbmat{\mbs{F}_{1}\\\mbs{F}_{3}}}\in\R^{2n\times 2n}$ is invertible.
\end{lem}
\begin{proof}
	Let $f(x)=\frac{2(2b+a-3x)}{(b-a)^2}$ and $g(x):=\frac{6(b+a-2x)}{(b-a)^3}$. Then, we note that $\int_{a}^{b}f(x)dx=-\int_{a}^{b}g(x)(x-a)=1$,
	whereas $\int_{a}^{b}f(x)(x-a)dx=\int_{a}^{b}g(x) dx =0$.
	By definition of the function $\mbs{F}_{3}$, it follows that
	\begin{align*}
		\int_{a}^{b}\mbs{F}_{3}(x)dx
		&=P_{12}^T,	&
		\int_{a}^{b}(x-a)\mbs{F}_{3}(x)dx
		&=P_{22}^T.
	\end{align*}
	and therefore (by Defn.~\ref{defn:Gmat}) $G_{0,\mbs{F}_{3}}=\bmat{P_{12}^T&P_{22}^T}$.
	Since $P$ is a permutation matrix, $P^{-1}=P^T$, and it follows that
	\begin{equation*}
		G_{0,\mbs{F}_{3}}P
		=\bmat{P_{12}^T&P_{22}^T}P	
		=\bmat{0_{m\times 2n-m}&I_{m}}.
	\end{equation*}
	Using this result we find
	\begin{equation*}
		G_{\tlbmat{E_{1}\\0},\tlbmat{\mbs{F}_{1}\\\mbs{F}_{3}}}P		
		=\srbmat{G_{E_{1},\mbs{F}_{1}}\\G_{0,\mbs{F}_{3}}}P
		=\bmat{I_{2n-m}&M\\0_{m\times 2n-m}&I_{m}}.
	\end{equation*}
	Since the matrix $\smallbmat{I&M\\0&I}$ is invertible, and the matrix $P$ is invertible, the matrix $G_{\tlbmat{E_{1}\\0},\tlbmat{\mbs{F}_{1}\\\mbs{F}_{3}}}$ is invertible as well.
\end{proof}

By Lem.~\ref{lem:BC_extra}, given $2n-m$ independent boundary conditions $\{E_{1},\mbs{F}_{1}\}$ such that $G_{E_{1},\mbs{F}_{1}}$ is of full row rank, we can introduce $m$ additional integral constraints defined by $\mbs{F}_{3}$, such that the the matrix $G_{\tlbmat{E_{1}\\0},\tlbmat{\mbs{F}_{1}\\\mbs{F}_{3}}}\in\R^{2n\times 2n}$ is invertible. Note that this choice of $\mbs{F}_{3}$ is not unique, and in many cases (if $m=n$) it suffices to simply let $\mbs{F}_{3}=I_{n}$ -- see e.g. Example~\ref{ex:Neumann_Tmap}. 
Given any such $\mbs{F}_{3}$, we then have sufficient conditions to constrain each of the degrees of freedom $\smallbmat{\alpha\\\beta}\in\R^{2n}$ in the definition of $\mcl{T}_{1}$ in~\eqref{eq:Tmap_Green}, allowing us to define an inverse to $\partial_{x}^2:X_{\tlbmat{E_{1}\\0},\tlbmat{\mbs{F}_{1}\\\mbs{F}_{3}}}\to L_{2}^{n}$ as follows.

\begin{prop}\label{prop:Tmap_hat}
	Let $\smallbmat{E_{1}\\0}\in\R^{2n\times 4n}$ and $\smallbmat{\mbs{F}_{1}\\\mbs{F}_{3}}\in L_{2}^{2n\times n}$ be such that $G_{\tlbmat{E_{1}\\0},\tlbmat{\mbs{F}_{1}\\\mbs{F}_{3}}}$ is invertible, and define
	\begin{equation*}
		\mbs{T}_{1}(x,\theta):=\bmat{I_{n}&\!\!\!(x-a)I_{n}}G_{\tlbmat{E_{1}\\0},\tlbmat{\mbs{F}_{1}\\\mbs{F}_{3}}}^{-1} \mbs{H}_{\tlbmat{E_{1}\\0},\tlbmat{\mbs{F}_{1}\\\mbs{F}_{3}}}(\theta).
	\end{equation*}
	Further define $\mcl{T}_{1}:L_{2}^{n}[a,b]\to W_{2}^{2,n}[a,b]$ by
	\begin{align*}
		(\mcl{T}_{1}\mbf{v})(x)&=\int_{a}^{b}\mbs{T}_{1}(x,\theta)\mbf{v}(\theta)d\theta +\int_{a}^{x}(x-\theta)\mbf{v}(\theta)d\theta,			
	\end{align*}
	for $\mbf{v}\in L_{2}^{n}[a,b]$.
	Then the following statements hold:
	\begin{enumerate}
		\item For all $\mbf{u}\in X_{\tlbmat{E_{1}\\0},\tlbmat{\mbs{F}_{1}\\\mbs{F}_{3}}}$, 
		$\mbf{u}=\mcl{T}_{1}\partial_{x}^2\mbf{u}$.
		
		\item For all $\mbf{v}\!\in L_{2}^{n}$, $\mcl{T}_{1}\mbf{v}\!\in X_{\tlbmat{E_{1}\\0},\tlbmat{\mbs{F}_{1}\\\mbs{F}_{3}}}$, and $\mbf{v}=\partial_{x}^2 \mcl{T}_{1}\mbf{v}$.
	\end{enumerate}
\end{prop}
\begin{proof}
	The result follows by defining $\mcl{T}_{1}$ as in~\eqref{eq:Tmap_Green}, imposing the boundary conditions defined by $\{\smallbmat{E_{1}\\0},\smallbmat{\mbs{F}_{1}\\\mbs{F}_{3}}\}$ to find $\alpha(x-a)+\beta=\int_{a}^{b}\mbs{T}_{1}(x,\theta)\mbf{v}(\theta)d\theta$.
A full proof is given in Prop.~\ref{prop:Tmap_hat_appx} in Appx.~\ref{appx:proofs}.
\end{proof}

Defining $\mcl{T}_{1}$ as in Prop.~\ref{prop:Tmap_hat}, for any $\mbf{v}\in L_{2}^{n}$, we have $\mcl{T}_{1}\mbf{v}\in X_{E_{1},\mbs{F}_{1}}$, and $\partial_{x}^2\mcl{T}_{1}\mbf{v}=\mbf{v}$. By Cor.~\ref{cor:Yset}, it follows that for all $\mbf{v}\in \hat{Y}$, we have $\mcl{T}_{1}\mbf{v}\in X$ and $\partial_{x}^2\mcl{T}_{1}\mbf{v}=\mbf{v}$, wherefore $\mcl{T}_{1}$ defines a right-inverse to $\partial_{x}^2:X\to \hat{Y}$. Unfortunately, $\mcl{T}_{1}$ does not quite define a left-inverse to $\partial_{x}^2$. Indeed, by construction, $\mcl{T}_{1}\mbf{v}\in X_{0,\mbs{F}_{3}}$ for all $\mbf{v}\in L_{2}^{n}$, so that $\mcl{T}_{1}\partial_{x}^2\mbf{u}$ corresponds to a projection of $\mbf{u}\in X$ onto the subspace $X\cap X_{0,\mbs{F}_{3}}$. This leaves the question of how to recover arbitrary elements of $X$ from their projections onto $X\cap X_{0,\mbs{F}_{3}}$, which we resolve in the next subsection.

\subsection{A Bijection Between the PDE Domain and a Subspace of $\R^{m}\times L_{2}^{n}$}\label{sec:Tmap:Tmap}

In the previous subsection, it was shown how we can define an integral operator $\mcl{T}_{1}:\hat{Y}\to X$ that acts as a right-inverse to the differential operator $\partial_{x}^{2}:X\to \hat{Y}$, so that $\partial_{x}^2\circ\mcl{T}_{1}=I$. Moreover, the complementary operator $\mcl{T}_{1}\circ\partial_{x}^{2}$ acts as a projection, mapping elements of $X$ onto a subspace defined by the auxiliary constraint $\mcl{F}\mbf{u}:=\int_{a}^{b}\mbs{F}_{3}(x)\mbf{u}(x)dx=0$. In this subsection, we show how we can recover any $\mbf{u}\in X$ from its projection $\mcl{T}_{1}\partial_{x}^2\mbf{u}$ and the value $\mcl{F}\mbf{u}\in\R^{m}$, defining $\mcl{T}_{0}:\R^{m}\to L_{2}^{n}$ such that $\mbf{u}=\mcl{T}_{0}\mcl{F}\mbf{u}+\mcl{T}_{1}\partial_{x}^{2}\mbf{u}$. In order for this to be satisfied, we note that $\mcl{T}_{0}$ must be such that not only $\mcl{F}(\mcl{T}_{0}\mcl{F})\mbf{u}=\mcl{F}\mbf{u}$ but also $\partial_{x}^2(\mcl{T}_{0}\mcl{F})\mbf{u}=0$. The following lemma shows how this may be achieved.

\begin{lem}\label{lem:uhat}
	Let $E_{1}\in\R^{2n-m\times 4n}$, $\mbs{F}\in L_{2}^{2n-m\times n}$, and $\mbs{F}_{3}\in L_{2}^{m\times n}$ be such that $G_{\tlbmat{E_{1}\\0},\tlbmat{\mbs{F}_{1}\\\mbs{F}_{3}}}$ is invertible. Define
	\begin{equation*}
		\mbs{T}_{0}(x):=\bmat{I_{n}&(x-a)I_{n}}G_{\tlbmat{E_{1}\\0},\tlbmat{\mbs{F}_{1}\\\mbs{F}_{3}}}^{-1}\bmat{0_{2n-m\times m}\\I_{m}},
	\end{equation*}
	and define $\mcl{T}_{0}:\R^{m}\to L_{2}^{n}[a,b]$ and $\mcl{F}:L_{2}^{n}[a,b]\to\R^{m}$ by
	\begin{align*}
		(\mcl{T}_{0}v)(x)&:=\mbs{T}_{0}(x)v,\qquad
		\mcl{F}\mbf{u}:=\int_{a}^{b}\mbs{F}_{3}(x)\mbf{u}(x)dx.
	\end{align*}
	for $v\in\R^{m}$ and $\mbf{u}\in L_{2}^{n}[a,b]$. Then, for all $v\in \R^{m}$, we have $\mcl{T}_{0}v\in X_{E_{1},\mbs{F}_{1}}$, $\partial_{x}^2\mcl{T}_{0}v=0$, and $\mcl{F}\mcl{T}_{0}v=v$.
\end{lem}
\begin{proof}
	Fix arbitrary $v\in\R^{m}$, and let $\hat{\mbf{u}}:=\mcl{T}_{0}v$. Then, by definition of the function $\mbs{T}_{0}$, we have $\partial_{x}^2 \mbs{T}_{0}(x)=0$ and thus $\hat{\mbf{u}}_{xx}(x)=\partial_{x}^2 \mbs{T}_{0}(x)v=0$. Furthermore, we note that
	\begin{align*}
		&\slbmat{\hat{\mbf{u}}(a)\\\hat{\mbf{u}}_{x}(a)}
		=\slbmat{\mbs{T}_{0}(a)\\(\partial_{x}\mbs{T}_{0})(a)}v	
		=G_{\tlbmat{E_{1}\\0},\tlbmat{\mbs{F}_{1}\\\mbs{F}_{3}}}^{-1}\bmat{0_{2n-m\times m}\\I_{m}}v,
	\end{align*}
	and therefore, multiplying both sides by $G_{E_{1},\mbs{F}_{1}}=\smallbmat{I_{2n-m}\\0_{m\times 2n-m}}^T G_{\tlbmat{E_{1}\\0},\tlbmat{\mbs{F}_{1}\\\mbs{F}_{3}}}$, it follows that
	\begin{equation*}
		G_{E_{1},\mbs{F}_{1}}\slbmat{\hat{\mbf{u}}(a)\\\hat{\mbf{u}}_{x}(a)}
		=\slbmat{I_{2n-m}\\0_{m\times 2n-m}}^T \bmat{0_{2n-m\times m}\\I_{m}}v
		=0.
	\end{equation*}
	We find that $G_{E_{1},\mbs{F}_{1}}\smallbmat{\hat{\mbf{u}}(a)\\\hat{\mbf{u}}_{x}(a)}=0=\int_{a}^{b}\mbs{H}_{E_{1},\mbs{F}_{1}}(\theta)\hat{\mbf{u}}_{xx}(\theta)d\theta$, and therefore $\hat{\mbf{u}}\in X_{E_{1},\mbs{F}_{1}}$.
	Finally, by definition of $\mcl{F}$,
	\begin{align*}
		&\mcl{F}\hat{\mbf{u}}
		=\int_{a}^{b}\mbs{F}_{3}(x)\mbs{T}_{0}(x)v\,dx	\\
		&=\int_{a}^{b}\bbl(\mbs{F}_{3}(x)\bmat{I_{n}&(x-a)I_{n}}\bbr)dx\; G_{\tlbmat{E_{1}\\0},\tlbmat{\mbs{F}_{1}\\\mbs{F}_{3}}}^{-1}\bmat{0_{2n-m\times m}\\v}	\\
		&=G_{0,\mbs{F}_{3}}G_{\tlbmat{E_{1}\\0},\tlbmat{\mbs{F}_{1}\\\mbs{F}_{3}}}^{-1}\bmat{0_{2n-m\times m}\\v}
		=v.
	\end{align*}
\end{proof}

Defining $\mcl{T}_{0}$ as in Lem.~\ref{lem:uhat}, we can identify any $v_{0}\in\R^{m}$ with $\mbf{u}_{0}:=\mcl{T}_{0}v_{0}\in X$ such that $\mcl{F}\mbf{u}_{0}=v_{0}$ and $\partial_{x}^2\mbf{u}_{0}=0$. On the other hand, defining $\mcl{T}_{1}$ as in Prop.~\ref{prop:Tmap_hat}, we can identify any $\mbf{v}_{1}\in \hat{Y}$ with $\mbf{u}_{1}:=\mcl{T}_{1}\mbf{v}_{1}\in X$ such that $\mcl{F}\mbf{u}_{1}=0$ and $\partial_{x}^2\mbf{u}_{1}=\mbf{v}_{1}$.
Consequently, introducing the augmented differential operator $\mscr{D}:=\smallbmat{\mcl{F}\\\partial_{x}^2}$, we can define a unique map $\mcl{T}:\mscr{D}\mbf{u}\mapsto\mbf{u}$ by an integral operator $\mcl{T}:=\bmat{\mcl{T}_{0}&\mcl{T}_{1}}$. The following theorem proves that these operators $\mscr{D}$ and $\mcl{T}$ in fact define a bijection between $X$ and $Y:=\R^{m}\times\hat{Y}$.

\begin{thm}\label{thm:Tmap}
	For given $E\in\R^{2n\times 4n}$ and $\mbs{F}\in L_{2}^{2n\times n}$, let $\sclbmat{E_{1}\\E_{2}}$ and $\sclbmat{\mbs{F}_{1}\\\mbs{F}_{2}}$ be as in Lem.~\ref{lem:BC_split}. Define $X$ and $Y$ as in~\eqref{eq:Xset} and~\eqref{eq:Yset_K}, respectively, where we let
	\begin{equation*}
		(\mcl{K}\smallbmat{v_{0}\\\mbf{v}_{1}})=\int_{a}^{b}\mbs{H}_{E_{2},\mbs{F}_{2}}(x)\mbf{v}_{1}(x)dx.
	\end{equation*}
	Further let $\mbs{F}_{3}\in L_{2}^{m\times n}$ be such that $G_{\tlbmat{E_{1}\\0},\tlbmat{\mbs{F}_{1}\\\mbs{F}_{3}}}$ is invertible, and define associated $\mcl{T}_{1}$ and $\{\mcl{F},\mcl{T}_{0}\}$ as in Prop.~\ref{prop:Tmap_hat} and Lem.~\ref{lem:uhat}, respectively.
	Finally, define $\mscr{D}:=\sclbmat{\mcl{F} \\ \partial_{x}^2}$, and let
	\begin{equation*}
		\mcl{T}\mbf{v}:=\mcl{T}_{0}v_{0} +\mcl{T}_{1}\mbf{v}_{1},\qquad \forall\mbf{v}=\sclbmat{v_{0}\\\mbf{v}_{1}}\in \R^{m}\times L_{2}^{n}[a,b],
	\end{equation*}
	Then the following statements hold:
	\begin{enumerate}
		\item 
		For every $\mbf{u}\in X$, $\mscr{D}\mbf{u} \in Y$ and $\mbf{u}=\mcl{T}\mscr{D}\mbf{u}$.
		\item 
		For every $\mbf{v} \in Y$, $\mcl{T}\mbf{v} \in X$ and $\mbf{v}=\mscr{D}\mcl{T}\mbf{v}$.
	\end{enumerate}
\end{thm}
\begin{proof}
	To prove the first statement, fix arbitrary $\mbf{u}\in X$, and let $\mbf{v}=\mscr{D}\mbf{u}$. By Lem.~\ref{lem:BC_expansion} and Lem.~\ref{lem:BC_split}, $\mbf{u}\in X_{E_{1},\mbs{F}_{1}}\cap X_{E_{2},\mbs{F}_{2}}$. It follows by Cor.~\ref{cor:Yset} that $\mbf{u}_{xx}\in \hat{Y}$, and therefore $\mbf{v}\in Y$. Let now $\hat{\mbf{u}}:=\mbf{u}-\mcl{T}_{0}\mcl{F}\mbf{u}$, so that $\mbf{u}=\hat{\mbf{u}}+\mcl{T}_{0}\mcl{F}\mbf{u}$. By Lem.~\ref{lem:uhat}, $\mcl{T}_{0}\mcl{F}\mbf{u}\in X_{E_{1},\mbs{F}_{1}}$, so by linearity of the boundary conditions, $\hat{\mbf{u}}\in X_{E_{1},\mbs{F}_{1}}$. In addition, by that same lemma, $\partial_{x}^2\mcl{T}_{0}\mcl{F}\mbf{u}=0$ and $\mcl{F}\mcl{T}_{0}\mcl{F}\mbf{u}=\mcl{F}\mbf{u}$, wherefore $\hat{\mbf{u}}_{xx}=\mbf{u}_{xx}$ and $\mcl{F}\hat{\mbf{u}}=0$. It follows that $\hat{\mbf{u}}\in X\cap X_{0,\mbs{F}_{3}}$, and thus, by Prop.~\ref{prop:Tmap_hat}, $\hat{\mbf{u}}=\mcl{T}_{1}\hat{\mbf{u}}_{xx}=\mcl{T}_{1}\mbf{u}_{xx}$. We find $\mbf{u}=\mcl{T}_{1}\mbf{u}_{xx}+\mcl{T}_{0}\mcl{F}\mbf{u}=\mcl{T}\mbf{v}$.
	
	
	To prove the second statement, fix arbitrary $(v_{0},\mbf{v}_{1})\in Y$, so that $\mbf{v}_{1}\in \hat{Y}$. By Prop.~\ref{prop:Tmap_hat}, $\partial_{x}^2\mcl{T}_{1}\mbf{v}_{1}=\mbf{v}_{1}$, and $\mcl{T}_{1}\mbf{v}\in X_{E_{1},\mbs{F}_{1}}\cap X_{0,\mbs{F}_{3}}$, whence $\mcl{F}\mcl{T}_{1}\mbf{v}_{1}=0$. By Lem.~\ref{lem:uhat}, $\partial_{x}^2\mcl{T}_{0}v_{0}=0$, and $\mcl{T}_{0}v_{0}\in X_{E_{1},\mbs{F}_{1}}$, with $\mcl{F}\mcl{T}_{0}v_{0}=v_{0}$. Letting $\mbf{u}=\mcl{T}_{1}\mbf{v}_{1}+\mcl{T}_{0}v_{0}$, it follows by linearity that $\mcl{F}\mbf{u}=v_{0}$, and $\mbf{u}_{xx}=\mbf{v}_{1}$. Furthermore, by linearity of the boundary conditions, we also have $\mbf{u}\in X_{E_{1},\mbs{F}_{1}}$. Finally, since $\mbf{u}_{xx}=\mbf{v}_{1}\in \hat{Y}$, by Cor.~\ref{cor:Yset}, $\mbf{u}\in X_{E_{2},\mbs{F}_{2}}$, and so by Lem.~\ref{lem:BC_split}, $\mbf{u}\in X$. 
\end{proof}

By Thm.~\ref{thm:Tmap}, the space $X\subseteq W_{2}^{2,n}$ is isomorphic to $Y\subseteq L_{2}^{n}$, with the augmented differential operator $\mscr{D}:X\to Y$ admitting an inverse as the integral operator $\mcl{T}:Y\to X$.
Here, although elements $\mbf{v}\in Y$ do have to satisfy a functional constraint $\mcl{K}\mbf{v}=0$, they are free of the regularity and boundary constraints imposed upon $\mbf{u}\in X$.
In the following sections, we will exploit this fact to show how we can analyse stability of linear PDEs with domain $X$, by parameterizing a Lyapunov functional on $Y$.
Of course, since PDEs may also involve lower-order derivatives of the state, this requires defining a map $\mcl{R}:\mscr{D}\mbf{u}\mapsto\mbf{u}_{x}$ as well, for which we have the following corollary. 

\begin{cor}\label{cor:Rmap}
	Let $\smallbmat{E_{1}\\0}\in\R^{2n\times 4n}$ and $\smallbmat{\mbs{F}_{1}\\\mbs{F}_{3}}\in L_{2}^{2n\times n}$ satisfy the conditions of Thm.~\ref{thm:Tmap}, and define $\mscr{D}$ and $\mcl{T}$ as in that theorem. Let $\mbs{T}_{1}\in L_{2}^{n\times n}$ and $\mbs{T}_{0}\in L_{2}^{n\times m}$ be as in Prop.~\ref{prop:Tmap_hat} and Lem.~\ref{lem:uhat}, respectively, and define
	\begin{equation*}\Resize{\linewidth}{
		(\mcl{R}\mbf{v})(x):=\partial_{x}\mbs{T}_{0}(x)v_{0} \!+\!\int_{a}^{b}\!\!\partial_{x}\mbs{T}_{1}(x,\theta)\mbf{v}_{1}(\theta)d\theta \!+\!\int_{a}^{x}\!\!\mbf{v}_{1}(\theta)d\theta,}
	\end{equation*}
	for $\mbf{v}=(v_{0},\mbf{v}_{1})\in\R^{m}\times L_{2}^{n}[a,b]$.
	Then the following holds:
	\begin{enumerate}
		\item 
		For every $\mbf{u}\in X_{E_{1},\mbs{F}_{1}}$, we have $\partial_{x}\mbf{u}=\mcl{R}\mscr{D}\mbf{u}$.
		
		\item 
		For every $\mbf{v}\in \R^{m}\times L_{2}^{n}$, we have $\mcl{R}\mbf{v}=\partial_{x}(\mcl{T}\mbf{v})$.		
	\end{enumerate}
\end{cor}
\begin{proof}
	The result follows from Thm.~\ref{thm:Tmap}, by applying Leibniz' integral rule to the operator $\mcl{T}$. A full proof is given in Cor.~\ref{cor:Rmap_appx}, in Appx.~\ref{appx:proofs}.
\end{proof}

By Thm.~\ref{thm:Tmap} and Cor.~\ref{cor:Rmap}, for any $\mbf{u}\in X$, both $\mbf{u}$ and $\mbf{u}_{x}$ may be uniquely expressed in terms of a corresponding $\mscr{D}\mbf{u}\in Y$, using integral operators $\mcl{T}$ and $\mcl{R}$. The following example illustrates what the space $Y$ and the resulting operator $\mcl{T}$ look like for the simple case of periodic boundary conditions, letting $\mcl{F}\mbf{u}:=\int_{a}^{b}\mbf{u}(x)dx$.

\begin{example}\label{ex:Neumann_Tmap}
	To illustrate an application of Thm.~\ref{thm:Tmap}, consider the following Sobolev subspace constrained by periodic boundary conditions,
	\begin{equation*}\Resize{\linewidth}{
		X=\!\bl\{\mbf{u}\in W_{2}^{2,n}[-1,1]\,\bl|\, \mbf{u}(-1)=\mbf{u}(1),~\mbf{u}_{x}(-1)=\mbf{u}_{x}(1)\br\}.}
	\end{equation*}
	By Cor.~\ref{cor:Yset}, for any $\mbf{u}\in X$ we have $\mbf{u}_{xx}\in \hat{Y}$, where
	\begin{equation*}
		\hat{Y}:=\bbl\{\mbf{v}\in L_{2}^{n}[-1,1]\:\,\bbl|\:\, \int_{-1}^{1}\mbf{v}(x)dx=0\bbr\}.
	\end{equation*}
	To obtain a unique expression for $\mbf{u}$ in terms of $\mbf{u}_{xx}$, we introduce $\mcl{F}\mbf{u}:=\frac{1}{2}\int_{-1}^{1}\mbf{u}(x)dx$, setting $\mbs{F}_{3}=\frac{1}{2}$ in Thm.~\ref{thm:Tmap}. Then, defining
	\begin{equation*}\Resize{\linewidth}{
		(\mcl{T}_{1}\mbf{v})(x)\!:=\!\int_{-1}^{x}\!\!(x-\theta)\mbf{v}(\theta)d\theta -\!\!\int_{-1}^{1}\!\!\frac{(\theta\!-\!1)(1\!-\!\theta\!+\!2x)}{4}\mbf{v}(\theta)d\theta,}
	\end{equation*}
	it follows by Thm.~\ref{thm:Tmap} that $\mbf{u}\in X$ if and only if $\mbf{u}=\frac{1}{2}\int_{-1}^{1}\!\mbf{u}(x)dx+\mcl{T}_{1}\mbf{u}_{xx}$. Conversely, $(v_{0},\mbf{v}_{1})\in Y:=\R\times \hat{Y}$ if and only if $v_{0}=\frac{1}{2}\int_{-1}^{1}\!\mbf{u}(x)dx$ and $\mbf{v}_{1}=\mbf{u}_{xx}$, where $\mbf{u}=v_{0}+\mcl{T}_{1}\mbf{v}_{1}$. 
\end{example}

\section{A PIE Representation of Linear PDEs}\label{sec:PDE2PIE}

In the previous section, it was proven that any $\mbf{u}\in X\subseteq W_{2}^{2,n}[a,b]$ can be uniquely identified by a corresponding element $\mbf{v}\in Y$, free of regularity and boundary constraints. In this section, we use this relation to show that for any $\mbf u(t)\in X$ satisfying a given form of linear PDE, we may uniquely associate a $\mbf{v}(t)\in Y$ that satisfies a corresponding Partial Integral Equation (PIE). To start, we first recall the definition and some properties of partial integral operators, used to parameterize such PIEs.

\subsection{PI Operators on $\R^{m}\times L_{2}^{n}[a,b]$}

PIEs are parameterized by Partial Integral (PI) operators. PI operators are an algebra of bounded, linear operators acting on $\R^{m}\times L_{2}^{n}$, parameterized by polynomials. Specifically, we define the set of 4-PI operators as follows.

\begin{defn}[4-PI operator, $\PIset^{(m,n)\times(p,q)}$]
	We say that $\mcl{P}:\R^{p}\times L_{2}^{q}\to \R^{m}\times L_{2}^{n}$ is a \textit{4-PI operator}, writing $\mcl{P}\in\PIset^{(m,n)\times(p,q)}$, if for some $P\in\R^{m\times p}$, $\mbs{Q}_{1}\in\R^{m\times q}[x]$, $\mbs{Q}_{2}\in\R^{n\times p}$, $\mbs{R}_{0}\in\R^{n\times q}[x]$ and $\mbs{R}_{1},\mbs{R}_{2}\in\R^{n\times q}[x,\theta]$, \\[-1.4em]
	\begin{align*}
		\bbl(\mcl{P}{\slbmat{v_{0}\\\mbf{v}_{1}}}\bbr)(x)
		:=\lbmat{Pv_{0}+\int_{a}^{b}\mbs{Q}_{1}(x)\mbf{v}_{1}(x)\\\mbs{Q}_{2}(x)v_{0}+\mbs{R}_{0}(x)\mbf{v}_{1}(x)+(\mcl{R}\mbf{v}_{1})(x)},
	\end{align*}
	for $(v_{0},\mbf{v}_{1})\in\R^{p}\times L_{2}^{q}[a,b]$, where \\[-1.6em]
	\begin{align*}
		(\mcl{R}\mbf{v}_{1})(x)&:={\int_{a}^{x}}\mbs{R}_{1}(x,\theta)\mbf{v}_{1}(\theta)d\theta +{\int_{a}^{b}}\mbs{R}_{2}(x,\theta)\mbf{v}_{1}(\theta)d\theta.
	\end{align*}
\end{defn}

\smallskip

By definition of the operators in Thm.~\ref{thm:Tmap} and Cor.~\ref{cor:Rmap}, if the space $X$ is defined by \textit{polynomial} $\mbs{F}$, then each of the operators $\{\mcl{T},\mcl{R},\mcl{F},\mcl{K}\}$ defined in the previous section is a 4-PI operator.
This will allow us to construct a PIE representation of linear PDEs in the following subsection, and test stability in this PIE representation in Section~\ref{sec:LPI}, by exploiting the algebraic properties of PI operators. Specifically, we note that 4-PI operators form a $*$-algebra, satisfying:
\begin{itemize}
	\item For any $\lambda,\mu\in\R$ and $\mcl{A},\mcl{B}\in\PIset^{(m,n)\times(p,q)}$, we have $\lambda\mcl{A}+\mu\mcl{B}\in\PIset^{(m,n)\times(p,q)}$;
	
	\item For any $\mcl{A}\in\PIset^{(m,n)\times(k,\ell)}$ and $\mcl{B}\in\PIset^{(k,\ell)\times(p,q)}$, we have $\mcl{A}\circ\mcl{B}\in\PIset^{(m,n)\times(p,q)}$;
	
	\item For any $\mcl{A}\in\PIset^{(m,n)\times(p,q)}$, $\mcl{A}^*\in\PIset^{(p,q)\times(m,n)}$.
\end{itemize}
Explicit parameter maps defining these operations can be found in~\cite{shivakumar2024GPDE}, and these operations can also be performed with the Matlab toolbox PIETOOLS~\cite{shivakumar2025PIETOOLS2024}, using overloaded functions for the sum (\texttt{+}), product (\texttt{*}), and transpose (\texttt{'}).

\subsection{Converting a PDE to a PIE}
Using the class of PI operators from the previous subsection, we now propose an equivalent representation of linear PDEs that is parameterized by such PI operators, as a PIE.
For the sake of clarity, we will consider only 2nd-order, linear PDEs here, although results in this section can be generalized to higher-order PDEs, coupled ODE-PDEs, and nonlinear PDEs as well.
Specifically, consider a 1D PDE of the form
\begin{align}\label{eq:PDE_standard}
	\mbf{u}_{t}(t,x)&=\mbs{A}_{0}(x)\mbf{u}(t,x) +\!\mbs{A}_{1}(x)\mbf{u}_{x}(t,x) +\!\mbs{A}_{2}(x)\mbf{u}_{xx}(t,x),	\notag\\
	\mbf{u}(t)&\in X,\hspace*{4.0cm} t\geq 0,	
\end{align}
parameterized by $\mbs{A}_{0},\mbs{A}_{1},\mbs{A}_{2}\in \R^{n\times n}[x]$,  $E\in\R^{2n\times 4n}$ and $\mbs{F}\in\R^{2n\times n}[x]$. We define solutions to the PDE as follows.
\begin{defn}
	For given $\mbs{A}_{0},\mbs{A}_{1},\mbs{A}_{2}\in \R^{n\times n}[x]$,  $E\in\R^{2n\times 4n}$ and $\mbs{F}\in\R^{2n\times n}[x]$, we say that $\mbf{u}$ is a solution to the PDE defined by $\{\mbs{A}_{i},E,\mbs{F}\}$ with initial state $\mbf{u}^{0}\in X$ if $\mbf{u}(t)$ is Fr\'echet differentiable, $\mbf{u}(0)=\mbf{u}^{0}$, and $\mbf{u}(t)$ satisfies~\eqref{eq:PDE_standard}, with $X$ as in~\eqref{eq:Xset}.
\end{defn}

Consider now a solution $\mbf{u}$ to the PDE~\eqref{eq:PDE_standard}. By Thm.~\ref{thm:Tmap} and Cor.~\ref{cor:Rmap}, the state $\mbf{u}(t)\in X$ at every time $t\geq 0$ must satisfy $\mbf{u}(t)=\mcl{T}\mbf{v}(t)$ and $\mbf{u}_{x}(t)=\mcl{R}\mbf{v}(t)$ for $\mbf{v}=\mscr{D}\mbf{u}$. Defining then the PI operator $\mcl{A}\in\PIset^{(0,n)\times(m,n)}$ as
\begin{equation}\label{eq:Aop}
	\mcl{A}:=\tnf{M}_{\mbs{A}_{0}}\circ\mcl{T} +\tnf{M}_{\mbs{A}_{1}}\circ\mcl{R} +\tnf{M}_{\mbs{A}_{2}}\circ\mcl{I},
\end{equation}
where we let
 $\mcl{I}\smallbmat{v_{0}\\\mbf{v}_{1}}:=\mbf{v}_{1}$, it follows that if $\mbf{u}(t)$ satisfies the PDE~\eqref{eq:PDE_standard}, then $\mbf{v}(t)=\mscr{D}\mbf{u}(t)=\smallbmat{\mcl{F}\mbf{u}\\\mbf{u}_{xx}}$ satisfies the PIE
\begin{equation}\label{eq:PIE_intermediate}
	\partial_{t}\mcl{T}\mbf{v}(t)
	=\mcl{A}\mbf{v}(t),\qquad \mbf{v}(t)\in Y,\quad t\geq0.
\end{equation}
In this representation, the \textit{fundamental state} $\mbf{v}(t)\in Y\subseteq\R^{m}\times L_{2}^{n}$ consists of $m+n$ components, but the dynamics are governed by only $n$ equations. However, recall from Thm.~\ref{thm:Tmap} that $\mcl{F}\mcl{T}\smallbmat{v_{0}\\\mbf{v}_{1}}=v_{0}$ for all $\smallbmat{v_{0}\\\mbf{v}_{1}}\in Y$. Applying $\mcl{F}$ to both sides of the PIE in~\eqref{eq:PIE_intermediate}, then, it follows that $v_{0}(t)$ must satisfy
\begin{equation*}
	\dot{v}_{0}(t)=\partial_{t}(\mcl{F}\circ\mcl{T})\slbmat{v_{0}(t)\\\mbf{v}_{1}(t)}=(\mcl{F}\circ\mcl{A})\slbmat{v_{0}(t)\\\mbf{v}_{1}(t)},\quad t\geq 0.
\end{equation*}
Finally, defining $\hat{\mcl{T}}:=\sclbmat{I_{m}\\\mcl{T}}$ and $\hat{\mcl{A}}:=\sclbmat{\mcl{F}\circ\mcl{A}\\\mcl{A}}$, we obtain a compact representation for the dynamics of $\mbf{v}=\smallbmat{v_{0}\\\mbf{v}_{1}}$ as
\begin{align}\label{eq:PIE_standard}
	\partial_{t}\hat{\mcl{T}}\mbf{v}(t)&=\hat{\mcl{A}}\mbf{v}(t),	&	\mbf{v}(t)\in Y,\quad t\geq 0.
\end{align}
We define solutions to this PIE as follows.
\begin{defn}
	For given $\mcl{T},\mcl{A}\in\PIset^{(m,n)\times(m,n)}$ and $\mcl{K}\in\PIset^{(m,0)\times(m,n)}$, we say that $\mbf{v}$ solves the PIE defined by $\{\mcl{T},\mcl{A},\mcl{K}\}$ with initial state $\mbf{v}^{0}\in Y$ for $Y$ as in~\eqref{eq:Yset_K} if $\mbf{v}(t)$ is Fr\'echet differentiable, $\mbf{v}(0)=\mbf{v}^{0}$, and $\mbf{v}(t)$ satisfies~\eqref{eq:PIE_standard}.
\end{defn}

The following result shows that the PDE and PIE representations are indeed equivalent, in that for any solution to the PDE we can define a solution to the PIE, and vice versa.

\begin{thm}\label{thm:PDE2PIE}
	For given $\mbs{A}_{0},\mbs{A}_{1},\mbs{A}_{2}\in\R^{n\times n}[x]$, $E\in\R^{2n\times 4n}$ and $\mbs{F}\in\R^{2n\times n}[x]$, define associated operators $\{\mcl{T},\mcl{F},\mcl{K}\}$ and $\mscr{D}$ as in Thm.~\ref{thm:Tmap}, and $\mcl{A}$ as in~\eqref{eq:Aop}, and let $\hat{\mcl{T}}:=\sclbmat{I_{m}\\\mcl{T}}$ and $\hat{\mcl{A}}:=\sclbmat{\mcl{F}\circ\mcl{A}\\\mcl{A}}$.
	Then, the following holds:
	\begin{enumerate}
		\item If $\mbf{u}$ solves the PDE defined by $\{\mbs{A}_{i},E,\mbs{F}\}$ with initial state $\mbf{u}^{0}\in X$, then $\mbf{v}:=\mscr{D}\mbf{u}$ solves the PIE defined by $\{\hat{\mcl{T}},\hat{\mcl{A}},\mcl{K}\}$ with initial state $\mbf{v}^{0}:=\mscr{D}\mbf{u}^{0}$.
		
		\item If $\mbf{v}$ solves the PIE defined by $\{\hat{\mcl{T}},\hat{\mcl{A}},\mcl{K}\}$ with initial state $\mbf{v}^{0}$, then $\mbf{u}:=\mcl{T}\mbf{v}$ solves the PDE defined by $\{\mbs{A}_{i},E,\mbs{F}\}$ with initial state $\mbf{u}^{0}:=\mcl{T}\mbf{v}^{0}$.
	\end{enumerate} 
\end{thm}
\begin{proof}
	The result follows from the fact that, by Thm.~\ref{thm:Tmap} and Cor.~\ref{cor:Rmap}, any $\mbf{u}\in X$ satisfies $\mbf{u}=\mcl{T}\mbf{v}$ and $\mbf{u}_{x}=\mcl{R}\mbf{v}$ for $\mbf{v}=\mscr{D}\mbf{u}\in Y$, and conversely, any $\mbf{v}\in Y$ satisfies $\mbf{v}=\mscr{D}\mbf{u}$, $\mcl{R}\mbf{v}(t)=\mbf{u}_{x}$, for $\mbf{u}=\mcl{T}\mbf{v}\in X$. 
	By definition of $\hat{\mcl{T}}$ and $\hat{\mcl{A}}$, it then follows that $\mbf{u}$ satisfies~\eqref{eq:PDE_standard} if and only if $\mbf{v}$ satisfies~\eqref{eq:PIE_standard}.
	A full proof is provided in Thm.~\ref{thm:PDE2PIE_appx} in Appx.~\ref{appx:proofs}.
\end{proof}

By Thm.~\ref{thm:PDE2PIE}, any PDE of the form in~\eqref{eq:PDE_standard} admits an equivalent PIE representation of the form in~\eqref{eq:PIE_standard}, so that for any solution to either system we can define a solution to the other. The following example illustrates what this PIE representation looks like for a simple reaction-diffusion PDE.


\begin{example}\label{ex:Neumann_PIE}
	To illustrate an application of Thm.~\ref{thm:PDE2PIE}, consider the reaction-diffusion equation
	\begin{align}\label{eq:Ex_reaction_diffusion_PDE}
		\mbf{u}_{t}(t,x)&=\mbf{u}_{xx}(t,x) +\lambda\mbf{u}(t,x),	& &\hspace*{-1cm}x\in(-1,1),\\
		\mbf{u}(t,-1)&=\mbf{u}(t,1),\quad\mbf{u}_{x}(t,-1)=\mbf{u}_{x}(t,1),	&	t&\geq 0.		\notag
	\end{align}
	Let $\mcl{F}\mbf{u}:=\frac{1}{2}\int_{-1}^{1}\mbf{u}(x)dx$, and define the operator $\mcl{T}_{1}$ as in Ex.~\ref{ex:Neumann_Tmap}. Then, by Thm.~\ref{thm:Tmap}, $\mbf{u}$ is a solution to the PDE~\eqref{eq:Ex_reaction_diffusion_PDE} if and only if $(v_{0},\mbf{v}_{1}):=(\mcl{F}\mbf{u},\mbf{u}_{xx})$ is a solution to the PIE
	\begin{align}\label{eq:Ex_reaction_diffusion_PIE}
		\dot{v}_{0}(t)&=\lambda v_{0}(t),	\\
		\partial_{t}(v_{0}(t)+\mcl{T}_{1}\mbf{v}_{1}(t))&=\mbf{v}_{1}(t)+\lambda(v_{0}(t)+\mcl{T}_{1}\mbf{v}_{1}(t)),	\notag\\
		{\textstyle \int_{-1}^{1}}\mbf{v}_{1}(t,x)dx&=0.	\notag
	\end{align}
	Conversely, $(v_{0},\mbf{v}_{1})$ will satisfy the PIE~\eqref{eq:Ex_reaction_diffusion_PIE} if and only if $\mbf{u}:=v_{0}+\mcl{T}_{1}\mbf{v}_{1}$ satisfies the PDE~\eqref{eq:Ex_reaction_diffusion_PDE}.
	Note here that, for any PDE solution $\mbf{u}$, the trajectory $\mbf{u}^*(t):=\mcl{F}\mbf{u}(t)=v_{0}(t)$ is a solution to the PDE as well, satisfying $\mbf{u}_{xx}(t)=\mbf{v}_{1}(t)=0$. In the next section, we show how we can verify stability of this trajectory, in the sense that $\lim_{t\to\infty}\norm{\mbf{u}(t)-\mbf{u}^*(t)}=0$.
\end{example}

\section{Stability Analysis in the PIE Representation}\label{sec:LPI}

In the previous section, it was shown how a broad class of linear, 1D PDEs can be equivalently represented as PIEs. In this section, we use this equivalence to show how stability of PDE trajectories can be tested by solving an operator inequality on PI operator variables. In order to do so, however, we must first define a suitable notion of stability. For example, for the reaction-diffusion equation in Ex.~\ref{ex:Neumann_PIE}, solutions $\mbf{u}(t)$ do not converge to any global equilibrium, but do all converge asymptotically to some uniform function $\mcl{F}\mbf{u}(t):=\frac{1}{2}\int_{-1}^{1}\mbf{u}(t,x)dx$ -- i.e. $\lim_{t\to\infty}\|\mbf{u}(t)-\mcl{F}\mbf{u}(t)\|_{L_{2}}=0$ (see Appx.~\ref{appx:reaction_diffusion}). We define a notion of exponential stability that encapsulates this kind of convergence as follows.

\begin{defn}\label{defn:stability}
	For a PDE defined by $\tnf{G}_{\tnf{pde}}=\{\mbs{A}_{i},E,\mbs{F}\}$ and an operator $\mcl{S}:L_{2}^{n}\to L_{2}^{n}$, we say that trajectories defined by $\mcl{S}$ are exponentially stable with rate $\alpha\geq 0$ if there exists some $M\in\R$ such that
	\begin{equation*}
		\norm{\mbf{u}(t)-\mcl{S}\mbf{u}(t)}_{L_{2}}\leq Me^{\alpha t}\norm{\mbf{u}(0)-\mcl{S}\mbf{u}(0)}_{L_{2}},\quad \forall t\geq 0,
	\end{equation*}
	for all solutions $\mbf{u}$ to the PDE defined by $\tnf{G}_{\tnf{pde}}$,
\end{defn}

Defn.~\ref{defn:stability} offers a definition of stability that does not require all solutions to converge to a unique equilibrium point.
Instead, the choice of $\mcl{S}$ offers substantial freedom in testing desired convergence behaviour of solutions, with e.g. $\mcl{S}=I_{n}$ posing no constraints on solutions, whilst $\mcl{S}=0$ requires global exponential stability of $\mbf{u}^*\equiv0$.
Given the decomposition $\mbf{u}(t)=\mcl{T}_{0}\mcl{F}_{0}\mbf{u}(t)+\mcl{T}_{1}\mbf{u}_{xx}(t)$ of the PDE solutions, an obvious, moderate choice is $\mcl{S}=\mcl{T}_{0}\circ\mcl{F}$, requiring $\lim_{t\to\infty}\norm{\mcl{T}_{1}\mbf{u}_{xx}(t)}_{L_{2}}=0$ even if $\lim_{t\to\infty}\norm{\mbf{u}(t)}_{L_{2}}\neq0$.

Having parameterized a notion of stability by operators $\mcl{S}$, we now consider how we may actually verify stability for a given choice of such $\mcl{S}$. For simplicity, let $\mcl{S}=0$ (testing exponential stability of $\mbf{u}^*=0$), and parameterize a candidate Lyapunov functional $V(\mbf{u})=\ip{\mbf{u}}{\mcl{P}\mbf{u}}=\ip{\mcl{T}\mbf{v}}{\mcl{P}\mcl{T}\mbf{v}}$ by a PI operator $\mcl{P}$. Taking the Lie derivative $\dot{V}$ of this functional along solutions to the PIE~\eqref{eq:PIE_intermediate}, we find
\begin{equation*}
	\dot{V}(\mcl{T}\mbf{v})=\ip{\mbf{v}}{[\mcl{T}^*\mcl{P}\mcl{A}+\mcl{A}^*\mcl{P}\mcl{T}]\mbf{v}}_{L_{2}},	\qquad\forall \mbf{v}\in Y.
\end{equation*}
Here,
$\mcl{Q}:=\mcl{T}^*\mcl{P}\mcl{A}+\mcl{A}^*\mcl{P}\mcl{T}$ is again a PI operator, so that stability can be tested by solving the Linear PI Inequality (LPI) defined by $\mcl{P}\succ 0$ and $\mcl{Q}\preceq 0$. 
However, since $\dot{V}(\mcl{T}\mbf{v})\leq 0$ need only hold for $\mbf{v}\in Y:=\{\mbf{v}\in\R^{m}\times L_{2}\mid \mcl{K}\mbf{v}=0\}$, simply imposing $\mcl{Q}\preceq 0$ may introduce significant conservatism.
The challenge, then, is to enforce $\ip{\mbf{v}}{\mcl{Q}\mbf{v}}_{L_{2}}\leq 0$ only for $\mbf{v}$ satisfying $\mcl{K}\mbf{v}=0$. The following proposition shows how this can be achieved using an approach similar to Finsler's lemma, introducing a slack operator variable $\mcl{X}$.

\begin{prop}\label{prop:LPI}
	For any $\tnf{G}_{\tnf{pde}}=\{\mbs{A}_{i},E,\mbs{F}\}$, let associated operators $\tnf{G}_{\tnf{pie}}:=\{\mcl{T},\mcl{A},\mcl{K}\}$ be as in Thm.~\ref{thm:PDE2PIE}. Let $\mcl{S}\in\PIset^{(0,n)\times(0,n)}$, and define $\tilde{\mcl{T}}:=(I_{n}-\mcl{S})\circ \mcl{T}$ and $\tilde{\mcl{A}}:=(I_{n}-\mcl{S})\circ \mcl{T}$.
	If there exist $\epsilon>0$, $\alpha\geq 0$, $\mcl{P}\in\PIset^{(0,n)\times (0,n)}$ and $\mcl{X}\in\PIset^{(0,n)\times(m,0)}$ such that
	\begin{align}\label{eq:stability_LPI}
		&\mcl{P}=\mcl{P}^*\succeq\epsilon^2 I,	\\
		&\tilde{\mcl{A}}^*\mcl{P}\tilde{\mcl{T}} 		+\tilde{\mcl{T}}^*\mcl{P}\tilde{\mcl{A}}\preceq -2\alpha\tilde{\mcl{T}}^*\mcl{P}\tilde{\mcl{T}}-\mcl{X}\mcl{K}-\mcl{K}^*\mcl{X}^*,\notag
	\end{align}
	then, trajectories defined by $\mcl{S}$ are exponentially stable with rate $\alpha$ for the PDE defined by $\tnf{G}_{\tnf{pde}}$.
\end{prop}
\begin{proof}
	Consider the functional $V:\R^{m}\times L_{2}^{n}[a,b]\to \R$ defined by $V(\mbf{v}):=\langle\tilde{\mcl{T}}\mbf{v},\mcl{P}\tilde{\mcl{T}}\mbf{v}\rangle_{L_{2}}$ for $\mbf{v}\in\R^{m}\times L_{2}^{n}$.
	Since $\mcl{P}\succeq\epsilon^2 I$, this function is bounded as
	\begin{equation*}
		\epsilon^2\bl\|\tilde{\mcl{T}}\mbf{v}\br\|_{L_{2}}^2\leq V(\mbf{v})\leq \mu^2\bl\|\tilde{\mcl{T}}\mbf{v}\br\|_{L_{2}}^2,
	\end{equation*}
	where $\mu:=\norm{\mcl{P}}_{\tnf{op}}^{1/2}$.
	Now, let $\mbf{u}$ be an arbitrary solution to the PDE defined by $\tnf{G}_{pde}$, and let $\mbf{v}:=\mscr{D}\mbf{u}$. Then, by Thm.~\ref{thm:PDE2PIE}, $\mbf{v}$ is a solution to the PIE defined by $\tnf{G}_{\tnf{pie}}$, and $\mbf{u}=\mcl{T}\mbf{v}$. It follows that $\mbf{v}$ satisfies
	\begin{equation*}
		\partial_{t}\tilde{\mcl{T}}\mbf{v}(t)
		=(I_{n}-\mcl{S})\partial_{t}\mcl{T}\mbf{v}(t)
		=(I_{n}-\mcl{S})\mcl{A}\mbf{v}(t)
		=\tilde{\mcl{A}}\mbf{v}(t),
	\end{equation*}
	and therefore
	\begin{align*}
		&\partial_{t}V(\mbf{v}(t))	
		=\bl\langle\partial_{t}\tilde{\mcl{T}}\mbf{v}(t),\mcl{P}\tilde{\mcl{T}}\mbf{v}(t)\br\rangle_{L_{2}} \!\!+\bl\langle\tilde{\mcl{T}}\mbf{v}(t),\mcl{P}\partial_{t}\tilde{\mcl{T}}\mbf{v}(t)\br\rangle_{L_{2}}	\\
		&=\bl\langle\mbf{v}(t),[\tilde{\mcl{A}}^*\mcl{P}\tilde{\mcl{T}}+\tilde{\mcl{T}}^*\mcl{P}\tilde{\mcl{A}}]\mbf{v}(t)\br\rangle_{L_{2}}	\\
		&\leq -\bl\langle\mbf{v}(t),[2\alpha\tilde{\mcl{T}}^*\mcl{P}\tilde{\mcl{T}}+\mcl{X}\mcl{K}+\mcl{K}^*\mcl{X}^*]\mbf{v}(t)\br\rangle_{L_{2}}	\\
		&=-2\alpha\bl\langle\tilde{\mcl{T}}\mbf{v}(t),\mcl{P}\tilde{\mcl{T}}\mbf{v}(t)\rangle_{L_{2}}
		=-2\alpha V(\mbf{v}(t)),
	\end{align*}
	where we remark that $\mcl{K}\mbf{v}(t)=0$ since $\mbf{v}(t)\in Y$, for all $t\geq 0$. Applying the Gr\"onwall-Bellman inequality, we find $V(\mbf{v}(t))\leq V(\mbf{v}(0))e^{-2\alpha t}$,
	and therefore $\epsilon^2\|\tilde{\mcl{T}}\mbf{v}(t)\|_{L_{2}}^2
	\leq \mu^2\|\tilde{\mcl{T}}\mbf{v}(0)\|_{L_{2}}^2e^{-2\alpha t}$. Since $\tilde{\mcl{T}}\mbf{v}=(I_{n}-\mcl{S})\mcl{T}\mbf{v}=\mbf{u}-\mcl{S}\mbf{u}$,
	it follows that trajectories defined by $\mcl{S}$ are exponentially stable with rate $\alpha$ for the PDE defined by $\tnf{G}_{\tnf{pde}}$.
\end{proof}

Using Prop.~\ref{prop:LPI}, for any PDE as in~\eqref{eq:PDE_standard}, exponential stability (as per Defn.~\ref{defn:stability}) of trajectories defined by a PI operator $\mcl{S}$ can be tested by solving the LPI~\eqref{eq:stability_LPI} for the associated PIE representation.
In the following section, we will use this result to verify stability properties for two PDE examples.

\section{Numerical Examples}\label{sec:examples}

In this section, we apply the proposed methodology to verify stability of two example PDEs. In each case, the LPI~\eqref{eq:stability_LPI} is implemented in Matlab using PIETOOLS~\cite{shivakumar2025PIETOOLS2024}, numerically parameterizing $\mcl{P}$ using monomials of degree at most $2d=6$, and solving the resulting semidefinite program with Mosek~\cite{mosek}. We refer to~\cite{shivakumar2024GPDE} for more details on how LPIs can be solved using semidefinite programming.

\subsection{Reaction-Diffusion Equation with Periodic BCs}\label{sec:examples:reaction_diffusion}

Consider the reaction-diffusion equation from Ex.~\ref{ex:Neumann_PIE},\\[-1.6em]
\begin{align}\label{eq:Ex_reaction_diffusion_PDE_2}
	\mbf{u}_{t}(t,x)&=\mbf{u}_{xx}(t,x) +\lambda\mbf{u}(t,x),	& &\hspace*{-1cm}x\in(-1,1),\\
	\mbf{u}(t,-1)&=\mbf{u}(t,1),\quad\mbf{u}_{x}(t,-1)=\mbf{u}_{x}(t,1),	&	t&\geq 0.		\notag
	\\[-1.6em]		\notag
\end{align}
For any $\lambda\in[0,\pi^2)$, the trajectories defined by the operator $\mcl{S}\mbf{u}(t):=\frac{1}{2}\int_{-1}^{1}\mbf{u}(t,x)dx$ are exponentially stable for this PDE in the sense of Defn.~\ref{defn:stability}, with rate $\alpha=\pi^2-\lambda$  (see Appx.~\ref{appx:reaction_diffusion}). We can verify this using Prop.~\ref{prop:LPI}, setting $\mcl{T}=\bmat{I&\mcl{T}_{1}}$ and $\mcl{A}=\bmat{\lambda&I-\lambda\mcl{T}_{1}}$, for $\mcl{T}_{1}$ as in Ex.~\ref{ex:Neumann_Tmap}. Using PIETOOLS, exponential stability with rate $\alpha=0$ can be verified up to $\lambda=9.8695$ approaching the true stability limit $\lambda=\pi^2\approx 9.8696$. Fixing several values of $\lambda$, exponential stability can be verified with rates $\alpha$ as in Tab.~\ref{tab:decay_reaction_diffusion_mean0}, also approaching the true rates $\alpha_{\max}=\pi^2-\lambda$.

\begin{table}[h]
	\vspace*{-0.1cm}
	\setlength{\tabcolsep}{3pt}
	\begin{tabular}{c|cccccccc}
		$\lambda$ & 0 & 1.5 & 3 & 4.5 & 6 & 7.5 & 9 & 9.5 \\\hline
		$\alpha$ & 9.8690 & 8.3691 & 6.8692 & 5.3693 & 3.8695 & 2.3695 & 0.8696 & 0.3696 \\
		$\alpha_{\max}$ & 9.8696 & 8.3696 & 6.8696 & 5.3696 & 3.8696 & 2.3696 & 0.8696 & 0.3696
	\end{tabular}
	\caption{Rates $\alpha$ for which exponential stability of trajectories\\ defined by $\mcl{S}\mbf{u}:=\frac{1}{2}\int_{-1}^{1}\mbf{u}(x)dx$ for the PDE~\eqref{eq:Ex_reaction_diffusion_PDE_2} was verified\\[-0.3em] with Prop.~\ref{prop:LPI}, along with the analytic rate $\alpha_{\max}=\pi^2-\lambda$.}\label{tab:decay_reaction_diffusion_mean0}
	\vspace*{-0.5cm}
\end{table}


\subsection{Wave Equation with Neumann BCs}

Consider now the following wave equation with stabilizing feedback and Neumann boundary conditions, \\[-1.6em]
\begin{align}\label{eq:Ex_wave}
	\mbf{u}_{tt}(t,x)&=\!\mbf{u}_{xx}(t,x) -\!2k\mbf{u}_{t}(t,x)-\!k^2\mbf{u}(t,x),\enspace t\geq 0,	\\
	\mbf{u}_{x}(t,0)&=\mbf{u}_{x}(t,1)=0,\hspace*{3.2cm} x\in(0,1). \notag
	\\[-1.6em]\notag
\end{align}
To convert this PDE to a PIE representation, we introduce $\mbs{\phi}(t,x)=(\mbf{u}(t,x),\mbf{u}_{t}(t,x))^T$, expanding the system as\\[-1.2em]
\begin{align*}
	\mbs{\phi}_{t}(t,x)&={\small\bmat{0&1\\-k^2&\!\!-2k}}\mbs{\phi}(t,x)+{\small\bmat{0&\!0\\1&\!0}}\mbs{\phi}_{xx}(t,x),	\enspace t\geq 0. \\[-1.8em]
\end{align*}
A PIE representation can then be constructed using Thm.~\ref{thm:PDE2PIE}, introducing fundamental state components 
$v_{0}(t)=(\mcl{F}\mbs{\phi})(t):=\int_{0}^{1}(4-6x)\mbs{\phi}(t,x)dx$ (as in Lem.~\ref{lem:BC_extra}) and $\mbf{v}_{1}(t)=\mbs{\phi}_{xx}(t)$.
The trivial solution $\mbf{u}^*\equiv0$ can be shown to be exponentially stable if and only if $k\geq 0$, with rate $\alpha=k$ (see Appx.~\ref{appx:wave_eq}). Using PIETOOLS, exponential stability can be verified for $k\geq 0.0006$ with rate $\alpha=0$, and with rates $\alpha$ as in Tab.~\ref{tab:decay_wave} for several values of $k>0$.

\begin{table}[h]
	\setlength{\tabcolsep}{5pt}
	\vspace*{-0.1cm}
	\begin{tabular}{c|cccccccc}
		$k$ & 1 & 2 & 3 & 4 & 5 & 6 & 7 & 8 \\\hline
		$\alpha$ & 0.981 & 1.997 & 2.993 & 3.996 & 4.994 & 5.975 & 6.969 & 7.957
	\end{tabular}
	\caption{Maximal rates $\alpha$ for which exponential stability of $\mbs{\phi}^*\equiv 0$\\ ($S^*=0$ in Defn.~\ref{defn:stability}) for the PDE~\eqref{eq:Ex_wave} was verified with Prop.~\ref{prop:LPI},\\[-0.4em] for several values of $k$. The analytic rate is $\alpha_{\max}=k$.}\label{tab:decay_wave}
	\vspace*{-0.5cm}
\end{table}


\section{Conclusion}

In this paper, it was shown how for 2nd-order PDEs with periodic as well as more general linear boundary conditions, convergence of solutions to trajectories in the nullspace of $\partial_{x}^2$ can be tested.
To this end, it was first shown how for a PDE domain $X$ defined by such boundary conditions, we can define a functional $\mcl{F}:X\to\R^{m}$ and PI operator $\mcl{T}_{1}$, such that $\partial_{x}^2\circ\mcl{T}_{1}=I$ and $\mcl{F}\circ\mcl{T}_{1}=0$. Next, a PI operator $\mcl{T}_{0}:\R^{m}\to X$ was defined such that $\mbf{u}=\mcl{T}_{0}\mcl{F}\mbf{u}+\mcl{T}_{1}\mbf{u}_{xx}$ for any $\mbf{u}\in X$, where now $\partial_{x}^2\mcl{T}_{0}\mcl{F}_{0}\mbf{u}=0$.
This relation was then used to derive a PIE representation of linear PDEs in terms of $(\mcl{F}\mbf{u}(t),\mbf{u}_{xx}(t))$, and it was shown how exponential stability 
(in the sense of Defn.~\ref{defn:stability}) 
of both the trivial solution $\mbf{u}(t)=0$ and of trajectories $\mcl{T}_{0}\mcl{F}\mbf{u}(t)$ can be tested by solving an LPI. 
Results in this work may be readily combined with earlier results on the PIE methodology for e.g. estimator and controller synthesis, and stability analysis of nonlinear PDEs.

\bibliographystyle{IEEEtran}
\bibliography{bibfile}

\newpage

\begin{appendices}

\section{Proofs of Several Results}\label{appx:proofs}

In this appendix, we prove several of the results presented in Sec.~\ref{sec:Tmap} and Sec.~\ref{sec:PDE2PIE}. To start, recall that for $E\in\R^{m\times 4n}$ and $\mbs{F}\in L_{2}^{m\times n}[a,b]$ we define the Sobolev subspace
\begin{equation}\label{eq:Xset_appx}
	X\!:=\!\bbl\{\mbf{u}\in W_{2}^{2,n}\,\bbl|\, E\smallbmat{\mbf{u}(a)\\\mbf{u}(b)\\\mbf{u}_{x}(a)\\\mbf{u}_{x}(b)}+\int_{a}^{b}\!\!\mbs{F}(x)\mbf{u}(x)dx=0\bbr\}.
\end{equation}
We define an alternative representation of $X$ as follows.
\begin{defn}\label{defn:Gmat_appx}
	For given $E\in \R^{m\times 4n}$ and $\mbs{F}\in L_{2}^{m\times n}$, define $G_{E,\mbs{F}}\in \R^{m \times 2n}$ and $\mbs{H}_{E,\mbs{F}}\in L_2^{m\times n}$ by
	\begin{align*}
		&G_{E,\mbs{F}}:=E\scbrray{cc}{I_{n}&0_{n}\\I_{n}&(b-a)I_{n}\\0_{n}&I_{n}\\0_{n}&I_{n}}+\int_{a}^{b}\mbs{F}(x)\bmat{I_{n}&(x-a)I_{n}}dx,\\
		&\mbs{H}_{E,\mbs{F}}(x):=-E{\sclbmat{0_{n}\\(b-x)I_{n}\\0_{n}\\I_{n}}}-\int_{x}^{b}(\theta-x)\mbs{F}(\theta)d\theta,
	\end{align*}
	and define the subspace $X_{E,\mbs{F}}\subseteq W_{2}^{2,n}$ as
	\begin{align*}
		&X_{E,\mbs{F}}	\\[-0.6em]
		&~:= \bbl\{\mbf u\in W_{2}^{2,n}\,\bbl|\,		G_{E,\mbs{F}}
		{\slbmat{\mbf{u}(a)\\\mbf{u}_{x}(a)}}=\int_{a}^{b}\!\!\mbs{H}_{E,\mbs{F}}(x)\mbf{u}_{xx}(x)dx\bbr\}.
	\end{align*}
\end{defn}
The following lemma shows that, indeed, $X=X_{E,\mbs{F}}$
\begin{lem}\label{lem:BC_expansion_appx}
	For $E\in \R^{m\times 4n}$ and $\mbs{F}\in L_{2}^{m\times n}$, let $X$ be as in~\eqref{eq:Xset_appx}, and $X_{E,\mbs{F}}$ as in Defn.~\ref{defn:Gmat_appx}. Then $X=X_{E,\mbs{F}}$.
\end{lem}
\begin{proof}
	Fix arbitrary $\mbf{u}\in W_{2}^{2,n}[a,b]$. Then, by Taylor's theorem with integral form of the remainder,
	\begin{equation*}
		\mbf{u}(x)=\mbf{u}(a) +(x-a)\mbf{u}_{x}(a) +\int_{a}^{x}(x-\theta)\mbf{u}_{xx}(\theta)d\theta.
	\end{equation*}
	It follows that
	\begin{align*}
		\slbmat{\mbf{u}(b)\\\mbf{u}_{x}(b)}&=\slbmat{\mbf{u}(a) +(b-a)\mbf{u}_{x}(a) +\int_{a}^{b}(b-\theta)\mbf{u}_{xx}(\theta)d\theta	\\
			\mbf{u}_{x}(a) +\int_{a}^{b}\mbf{u}_{xx}(\theta)d\theta},
	\end{align*}
	and
	\begin{align*}
		&\int_{a}^{b}\mbs{F}(x)\mbf{u}(x)dx
		=\int_{a}^{b}\mbs{F}(x)\bbl(\mbf{u}(a)+(x-a)\mbf{u}_{x}(a)\bbr)\,dx	\\
		&\hspace*{2.8cm}+\int_{a}^{b}\mbs{F}(x)\left[\int_{a}^{x}(x-\theta)\mbf{u}_{xx}(\theta)d\theta\right] dx	\\
		&\qquad=\int_{a}^{b}\bmat{\mbs{F}(x)&\mbs{F}(x)(x-a)}dx\slbmat{\mbf{u}(a)\\\mbf{u}_{x}(a)}	\notag\\
		&\qquad\qquad +\int_{a}^{b}\left[\int_{\theta}^{b}(x-\theta)\mbs{F}(x)d\theta\right] \mbf{u}_{xx}(\theta)d\theta.
	\end{align*}
	Substituting these relations into the boundary conditions defining $X$ in~\eqref{eq:Xset_appx}, we find
	\begin{align*}
		&E{\sclbmat{\mbf{u}(a)\\\mbf{u}(b)\\\mbf{u}_{x}(a)\\\mbf{u}_{x}(b)}}+\int_{a}^{b}\mbs{F}(x)\mbf{u}(x)dx	\\[-0.4em]
		&\hspace*{1.5cm}=G_{E,\mbs{F}}{\slbmat{\mbf{u}(a)\\\mbf{u}_{x}(a)}}  -\int_{a}^{b}\mbs{H}_{E,\mbs{F}}(\theta)\mbf{u}_{xx}(\theta)d\theta,
	\end{align*}
	whence the result holds.
\end{proof}

Using this equivalent representation of the boundary conditions, we can show how, if $G_{E,\mbs{F}}$ is invertible, we can uniquely define an inverse of the differential operator $\partial_{x}^2:X\to L_{2}^{n}$ as a partial integral operator $\mcl{T}:L_{2}^{n}\to X$.

\begin{prop}\label{prop:Tmap_hat_appx}
	Let $E\in\R^{2n\times 4n}$ and $\mbs{F}\in L_{2}^{2n-m\times n}$ be such that $G_{E,\mbs{F}}$ is invertible, and define
	\begin{equation*}
		\mbs{T}_{1}(x,\theta):=\bmat{I_{n}&\!\!\!(x-a)I_{n}}G_{E,\mbs{F}}^{-1} \mbs{H}_{E,\mbs{F}}(\theta).
	\end{equation*}
	Further define $\mcl{T}_{1}:L_{2}^{n}[a,b]\to W_{2}^{2,n}[a,b]$ by
	\begin{align*}
		(\mcl{T}_{1}\mbf{v})(x)&=\int_{a}^{b}\mbs{T}_{1}(x,\theta)\mbf{v}(\theta)d\theta +\int_{a}^{x}(x-\theta)\mbf{v}(\theta)d\theta,			
	\end{align*}
	for $\mbf{v}\in L_{2}^{n}[a,b]$.
	Then the following statements hold:
	\begin{enumerate}
		\item For all $\mbf{u}\in X_{E,\mbs{F}}$, 
		$\mbf{u}=\mcl{T}_{1}\partial_{x}^2\mbf{u}$.
		
		\item For all $\mbf{v}\in L_{2}^{n}$, $\mcl{T}_{1}\mbf{v}\in X_{E,\mbs{F}}$, and $\mbf{v}=\partial_{x}^2 \mcl{T}_{1}\mbf{v}$.
	\end{enumerate}
\end{prop}
\begin{proof}
	To prove the first statement, fix arbitrary $\mbf{u}\in X_{E,\mbs{F}}$, and let $\mbf{v}=\mbf{u}_{xx}$. Since $G_{E,\mbs{F}}\in\R^{2n\times 2n}$ is invertible, and by definition of $X_{E,\mbs{F}}$, $\mbf{u}$ must satisfy
	\begin{equation*}
		\slbmat{\mbf{u}(a)\\\mbf{u}_{x}(a)}=\int_{a}^{b}G_{E,\mbs{F}}^{-1}\mbs{H}_{E,\mbs{F}}(\theta)\mbf{u}_{xx}(\theta)d\theta.
	\end{equation*}
	Using Taylor's theorem with integral form of the remainder, and invoking the definition of $\mbs{T}_{1}(x,\theta)$, it follows that
	\begin{align*}
		&\mbf{u}(x)=\bmat{I_{n}&(x-a)I_{n}}\slbmat{\mbf{u}(a)\\\mbf{u}_{x}(a)} +\int_{a}^{x}(x-\theta)\mbf{u}_{xx}(\theta)d\theta	\\
		&=\int_{a}^{b}\mbs{T}_{1}(x,\theta)\mbf{v}(\theta)d\theta +\int_{a}^{x}(x-\theta)\mbf{v}(\theta)d\theta
		=(\mcl{T}_{1}\mbf{v})(x).
	\end{align*}
	
	Now, to prove the second statement, fix arbitrary $\mbf{v}\in L_{2}^{n}$, and let $\mbf{u}=\mcl{T}_{1}\mbf{v}$. Then, by the Leibniz integral rule,
	\begin{align*}
		&\mbf{u}_{x}(x)=\int_{a}^{b}\partial_{x}\mbs{T}_{1}(x,\theta)\mbf{v}(\theta)d\theta +\int_{a}^{x}\mbf{v}(\theta)d\theta	\\
		&\mbf{u}_{xx}(x)=\int_{a}^{b}\partial_{x}^2\mbs{T}_{1}(x,\theta)\mbf{v}(\theta)d\theta
	\end{align*}
	Here, by definition of the function $\mbs{T}_{1}$, we have $\partial_{x}^2\mbs{T}_{1}(x,\theta)=0$, whence $\mbf{u}_{xx}(x)=0$. 
	Moreover, 
	\begin{equation*}
		\slbmat{\mbs{T}_{1}(a,\theta)\\\frac{\partial}{\partial x}\mbs{T}_{1}(a,\theta)}=G_{E,\mbs{F}}^{-1}\mbs{H}_{E,\mbs{F}}(\theta),
	\end{equation*}
	and therefore $\slbmat{\mbf{u}(a)\\\mbf{u}_{x}(a)}=\int_{a}^{b}G_{E,\mbs{F}}^{-1}\mbs{H}_{E,\mbs{F}}(\theta)\mbf{u}_{xx}(\theta)\,d\theta$.
	It follows that $\mbf{u}\in X_{E,\mbs{F}}$.
\end{proof}

By Prop.~\ref{prop:Tmap_hat_appx}, if $G_{E,\mbs{F}}$ is invertible, we can define a unique map $\mcl{T}_{1}:L_{2}^{n}\to X$ such that $\mcl{T}_{1}\mbf{u}_{xx}=\mbf{u}$ for $\mbf{u}\in X$, and $\partial_{x}^2\mcl{T}_{1}\mbf{v}=\mbf{v}$ for $\mbf{v}\in L_{2}^{n}$. However, if $G_{E,\mbs{F}}$ is not invertible, we will need to use the approach presented in Sec.~\ref{sec:Tmap} to define such a map from $\mbf{u}_{xx}$ back to $\mbf{u}$. In particular, using Lem.~\ref{lem:BC_split}, we can partition $X=X_{E_{1},\mbs{F}_{1}}\cap X_{E_{2},\mbs{F}_{2}}$ such that $G_{E_{1},\mbs{F}_{1}}\in\R^{2n-m\times 2n}$ is of full row rank whereas $G_{E_{2},\mbs{F}_{2}}=0$. Then, using Lem.~\ref{lem:BC_extra}, we can define $\mbs{F}_{3}$ such that $G_{\tlbmat{E_{1}\\0},\tlbmat{\mbs{F}_{1}\\\mbs{F}_{3}}}$ is invertible, allowing us to apply Prop.~\ref{prop:Tmap_hat_appx} to define an inverse $\mcl{T}_{1}$ to the differential operator. Finally, defining $\mcl{T}_{0}$ as in Lem.~\ref{lem:uhat}, we can express any $\mbf{u}\in X_{E_{1},\mbs{F}_{1}}$ as $\mbf{u}=\mcl{T}_{0}\mcl{F}\mbf{u}+\mcl{T}_{1}\partial_{x}^2\mbf{u}$, where $\mcl{F}\mbf{u}:=\int_{a}^{b}\mbs{F}_{3}(x)\mbf{u}(x)dx$, by Thm.~\ref{thm:Tmap}. The following corollary proves that we can also define the derivative $\mbf{u}_{x}$ in terms of $\mscr{D}\mbf{u}=\smallbmat{\mcl{F}\\\partial_{x}^2}\mbf{u}$, using an operator $\mcl{R}$.

\begin{cor}\label{cor:Rmap_appx}
	Let $E_{1}\in\R^{2n\times 4n}$, $\mbs{F}_{1}\in L_{2}^{2n-m\times n}$, and $\mbs{F}_{3}\in L_{2}^{m\times 2n}$ satisfy the properties of Thm.~\ref{thm:Tmap}, and define $\mscr{D}$ and $\mcl{T}$ as in that theorem. Let $\mbs{T}_{1}\in L_{2}^{n\times n}$ and $\mbs{T}_{0}\in L_{2}^{n\times m}$ be as in Prop.~\ref{prop:Tmap_hat} and Lem.~\ref{lem:uhat}, respectively, and define
	\begin{equation*}\Resize{\linewidth}{
			(\mcl{R}\mbf{v})(x):=\partial_{x}\mbs{T}_{0}(x)v_{0} \!+\!\int_{a}^{b}\!\!\partial_{x}\mbs{T}_{1}(x,\theta)\mbf{v}_{1}(\theta)d\theta \!+\!\int_{a}^{x}\!\!\mbf{v}_{1}(\theta)d\theta,}
	\end{equation*}
	for $\mbf{v}=(v_{0},\mbf{v}_{1})\in\R^{m}\times L_{2}^{n}[a,b]$.
	Then the following holds:
	\begin{enumerate}
		\item 
		For all $\mbf{u}\in X_{E_{1},\mbs{F}_{1}}$, we have $\partial_{x}\mbf{u}=\mcl{R}\mscr{D}\mbf{u}$.
		
		\item 
		For all $\mbf{v}\in \R^{m}\times L_{2}^{n}$, we have $\mcl{R}\mbf{v}=\partial_{x}(\mcl{T}\mbf{v})$.		
	\end{enumerate}
\end{cor}
\begin{proof}
	To prove both statements, note that by Leibniz' integral rule, we have for all $\mbf{v}\in \R^{m}\times L_{2}^{n}[a,b]$ that
	\begin{align}\label{eq:Rmap_proof}
		&\partial_{x} (\mcl{T}\mbf{v})(x)	\tag{*}\\
		&=\partial_{x}\!\left[\!\mbs{T}_{0}(x)v_{0}+\int_{a}^{b}\!\!\mbs{T}_{1}(x,\theta)\mbf{v}_{1}(\theta)d\theta +\int_{a}^{x}\!\!(x-\theta)\mbf{v}_{1}(\theta)d\theta\!\right]	\notag\\
		&=\partial_{x}\mbs{T}_{0}(x)v_{0}+\!\int_{a}^{b}\!\!\partial_{x}\mbs{T}_{1}(x,\theta)\mbf{v}_{1}(\theta)d\theta +\!\int_{a}^{x}\!\!\mbf{v}_{1}(\theta)d\theta \notag\\
		&=(\mcl{R}\mbf{v})(x).		\notag
	\end{align}
	Thus, the second statement holds.
	Now, for the first statement, fix arbitrary $\mbf{u}\in X_{E_{1},\mbs{F}_{1}}$. Then, by the proof of Thm.~\ref{thm:Tmap}, $\mbf{u}=\mcl{T}\mscr{D}\mbf{u}$, so invoking the identity in~\eqref{eq:Rmap_proof} it follows that $\mbf{u}_{x}=\mcl{R}\mscr{D}\mbf{u}$. 
\end{proof}

Using this corollary, as well as Thm.~\ref{thm:Tmap}, we can construct an equivalent PIE representation of a broad class of linear, 2nd order PDEs. In particular, we consider PDEs of the form
\begin{align}\label{eq:PDE_standard_appx}
	\mbf{u}_{t}(t,x)&=\mbs{A}_{0}(x)\mbf{u}(t,x) +\!\mbs{A}_{1}(x)\mbf{u}_{x}(t,x) +\!\mbs{A}_{2}(x)\mbf{u}_{xx}(t,x),	\notag\\
	\mbf{u}(t)&\in X,\hspace*{4.0cm} t\geq 0.
\end{align}
We define an associated PIE representation of this system as
\begin{align}\label{eq:PIE_standard_appx}
	\partial_{t}\hat{\mcl{T}}\mbf{v}(t)&=\hat{\mcl{A}}\mbf{v}(t),	&	\mbf{v}(t)\in Y,\quad t\geq 0.
\end{align}
where $\hat{\mcl{T}}:=\sclbmat{I_{m}\\\mcl{T}}$ and $\hat{\mcl{A}}:=\sclbmat{\mcl{F}\circ\mcl{A}\\\mcl{A}}$, with
\begin{equation}\label{eq:Aop_appx}
	\mcl{A}:=\tnf{M}_{\mbs{A}_{0}}\circ\mcl{T} +\tnf{M}_{\mbs{A}_{1}}\circ\mcl{R} +\tnf{M}_{\mbs{A}_{2}}\circ\mcl{I},
\end{equation}
where we let $\mcl{I}\smallbmat{v_{0}\\\mbf{v}_{1}}:=\mbf{v}_{1}$.
The following theorem proves that this PIE representation is equivalent to the PDE.

\begin{thm}\label{thm:PDE2PIE_appx}
	For given $\mbs{A}_{0},\mbs{A}_{1},\mbs{A}_{2}\in\R^{n\times n}[x]$, $E\in\R^{2n\times 4n}$ and $\mbs{F}\in\R^{2n\times n}[x]$, define associated operators $\{\mcl{T},\mcl{F},\mcl{K}\}$ and $\mscr{D}$ as in Thm.~\ref{thm:Tmap}, and $\mcl{A}$ as in~\eqref{eq:Aop_appx}, and let $\hat{\mcl{T}}:=\sclbmat{I_{m}\\\mcl{T}}$ and $\hat{\mcl{A}}:=\sclbmat{\mcl{F}\circ\mcl{A}\\\mcl{A}}$.
	Then, the following holds:
	\begin{enumerate}
		\item If $\mbf{u}$ solves the PDE defined by $\{\mbs{A}_{i},E,\mbs{F}\}$ with initial state $\mbf{u}^{0}\in X$, then $\mbf{v}=\mscr{D}\mbf{u}$ solves the PIE defined by $\{\hat{\mcl{T}},\hat{\mcl{A}},\mcl{K}\}$ with initial state $\mbf{v}^{0}=\mscr{D}\mbf{u}^{0}$.
		
		\item If $\mbf{v}$ solves the PIE defined by $\{\hat{\mcl{T}},\hat{\mcl{A}},\mcl{K}\}$ with initial state $\mbf{v}^{0}$, then $\mbf{u}:=\mcl{T}\mbf{v}$ solves the PDE defined by $\{\mbs{A}_{i},E,\mbs{F}\}$ with initial state $\mbf{u}^{0}=\mcl{T}\mbf{v}^{0}$.
	\end{enumerate} 
\end{thm}
\begin{proof}
	To prove the first statement, fix arbitrary $\mbf{u}^{0}\in X$, and let $\mbf{u}$ be an associated solution to the PDE~\eqref{eq:PDE_standard_appx}. Then $\mbf{u}(t)\in X\subseteq W_{2}^{2,n}[a,b]$ for all $t\geq 0$, and we can define $\mbf{v}(t)=\smallbmat{v_{0}(t)\\\mbf{v}_{1}(t)}:=\mscr{D}\mbf{u}(t)$. Clearly, then, $\mbf{v}^{0}=\mscr{D}\mbf{u}^{0}$. In addition, by Thm.~\ref{thm:Tmap} and Cor.~\ref{cor:Rmap_appx}, $\mbf{v}(t)\in Y$, and we have $\mbf{u}(t)=\mcl{T}\mbf{v}(t)$, $\mbf{u}_{x}(t)=\mcl{R}\mbf{v}(t)$ and $\mbf{u}_{xx}(t)=\mcl{I}\mbf{v}(t)$ for all $t\geq 0$. Since $\mbf{u}$ satisfies~\eqref{eq:PDE_standard_appx}, it follows that
	\begin{align*}
		&\partial_{t}\mcl{T}\mbf{v}(t)
		=\mbf{u}_{t}(t)	
		=\tnf{M}_{\mbs{A}_{0}}\mbf{u}(t) +\tnf{M}_{\mbs{A}_{1}}\mbf{u}_{x}(t) +\tnf{M}_{\mbs{A}_{2}}\mbf{u}_{xx}(t)	\\
		&\hspace*{1.0cm}=\bl(\tnf{M}_{\mbs{A}_{0}}\circ\mcl{T} +\tnf{M}_{\mbs{A}_{1}}\circ\mcl{R} +\tnf{M}_{\mbs{A}_{2}}\circ\mcl{I}\br)\mbf{v}(t)
		=\mcl{A}\mbf{v}(t).
	\end{align*}
	Applying the operator $\mcl{F}$ to both sides of this equation, and recalling that $\mcl{F}\circ \mcl{T}\smallbmat{v_{0}\\\mbf{v}_{1}}=v_{0}$, it follows that also $\dot{v}_{0}(t)=(\mcl{F}\circ\mcl{A})\mbf{v}(t)$.
	By definition of the operators $\hat{\mcl{T}}$ and $\hat{\mcl{A}}$, then, $\mbf{v}(t)$ is a solution to the PIE defined by $\{\hat{\mcl{T}},\hat{\mcl{A}},\mcl{K}\}$.
	
	For the second statement, fix arbitrary $\mbf{v}^{0}\in Y$, and let $\mbf{v}=(v_{0},\mbf{v}_{1})$ be an associated solution to the PIE~\eqref{eq:PIE_standard_appx}. Define $\mbf{u}(t)=\mcl{T}\mbf{v}(t)$ for all $t\geq 0$. Then, by Thm.~\ref{thm:Tmap} and Cor.~\ref{cor:Rmap_appx}, $\mbf{u}(t)\in X$, and we have $\mbf{u}_{x}(t)=\mcl{R}\mbf{v}(t)$ and $\mbf{u}_{xx}(t)=\mcl{I}\mbf{v}(t)$, for all $t\geq 0$. It follows that, for all $t\geq 0$,
	\begin{align*}
			\mbf{u}_{t}(t)
			=\partial_{t}\mcl{T}\mbf{v}(t)		
			&=\mcl{A}\mbf{v}(t)				\\
			&=\bl(\tnf{M}_{\mbs{A}_{0}}\circ\mcl{T} +\tnf{M}_{\mbs{A}_{1}}\circ\mcl{R} +\tnf{M}_{\mbs{A}_{2}}\circ\mcl{I}\br)\mbf{v}(t)		\\
			&=\tnf{M}_{\mbs{A}_{0}}\mbf{u}(t) +\tnf{M}_{\mbs{A}_{1}}\mbf{u}_{x}(t) +\tnf{M}_{\mbs{A}_{2}}\mbf{u}_{xx}(t).
		\end{align*}
	Thus $\mbf{u}$ is a solution to the PDE defined by $\{\mbs{A}_{i},E,\mbs{F}\}$.
\end{proof}

\section{Exponential Stability of PDE Examples}

\subsection{Reaction Diffusion Equation with Periodic BCs}\label{appx:reaction_diffusion}

Consider the reaction-diffusion equation with periodic boundary conditions,
\begin{align*}
	\mbf{u}_{t}(t,x)&=\mbf{u}_{xx}(t,x) +\lambda\mbf{u}(t,x),	&	&\hspace*{-1cm}x\in(-1,1),	\\
	\mbf{u}(t,-1)&=\mbf{u}(t,1),\quad \mbf{u}_{x}(t,-1)=\mbf{u}_{x}(t,1),						&	t&\geq 0.
\end{align*}
We can solve this equation using separation of variables. Specifically, let $\mbf{u}(t,x)=T(t)X(x)$. If $\mbf{u}(t,x)$ satisfies the PDE dynamics, then the functions $T$ and $X$ must satisfy
\begin{equation*}
	T'(t)X(x)=T(t)[X''(x) +\lambda X(x)]
\end{equation*}
and therefore
\begin{equation*}
	\frac{T'(t)}{T(t)}-\lambda =\frac{X''(x)}{X(x)}.
\end{equation*}
It follows that both sides must be equal to some constant $-c\in\R$, so that
\begin{equation*}
	T'(t)=[\lambda-c]T(t),
	\qquad
	X''(x)=-cX(x).
\end{equation*}
Invoking the periodic boundary conditions, we find that the ODE $X''(x)=-cX(x)$ admits non-trivial solutions only for $c=c_{n}=n^2\pi^2$ for $n\in\N_{0}$, taking the form
\begin{equation*}
	X_{0}(x):=\beta_{0},\quad X_{n}(x):=\alpha_{n}\sin(n\pi x)+\beta_{n}\cos(n\pi x),
\end{equation*}
for $\alpha_{n},\beta_{n}\in\R$.
For each $n\in\N_{0}$, we obtain a corresponding solution to the ODE $T'(t)=[\lambda-c_{n}]T(t)$ as
\begin{equation*}
	T_{n}(t)=\gamma_{n}e^{[\lambda-c_{n}]t}=\gamma_{n}e^{[\lambda-n^2\pi^2]t}
\end{equation*}
for $\gamma_{n}\in\R$. Thus, solutions to the PDE take the form 
\begin{align*}
	\mbf{u}(t,x)
	&=b_{0}e^{\lambda t}+\sum_{n=1}^{\infty}a_{n}e^{[\lambda-n^2\pi^2]t}\sin(n\pi x)	\\
	&\qquad +\sum_{n=1}^{\infty}b_{n}e^{[\lambda-n^2\pi^2]t}\cos(n\pi x)
\end{align*}
for coefficients $a_{n},b_{n}$ determined by the initial state $\mbf{u}(0)$. 
From this solution, it is clear that the origin $\mbf{u}\equiv 0$ for the PDE is stable if and only if $\lambda\leq 0$, with exponential decay rate $\alpha=\lambda$.
However, defining $\mcl{F}\mbf{u}(t):=\frac{1}{2}\int_{-1}^{1}\mbf{u}(t)$, we note that $\mcl{F}\mbf{u}(t)=b_{0}e^{\lambda t}$, and therefore
\begin{align*}
	\mbf{u}(t,x)-\mcl{F}\mbf{u}(t)
	&=\sum_{n=1}^{\infty}a_{n}e^{[\lambda-n^2\pi^2]t}\sin(n\pi x)	\\
	&\qquad +\sum_{n=1}^{\infty}e^{[\lambda-n^2\pi^2]t}b_{n}e^{[\lambda-n^2\pi^2]t}\cos(n\pi x).
\end{align*}
It follows that, for any $\lambda<\min_{n\in\N}n^2\pi^2=\pi^2$, solutions $\mbf{u}(t)$ do converge to the uniform state $\mcl{F}\mbf{u}(t)$, with exponential rate of decay $\alpha=\lambda-\min_{n\in\N}n^2\pi^2=\lambda-\pi^2$.

\subsection{Wave Equation with Neumann BCs}\label{appx:wave_eq}

Consider a modified wave equation with Neumann boundary conditions,
\begin{align*}
	\mbf{u}_{tt}(t,x)&=\mbf{u}_{xx}(t,x)-2k\mbf{u}_{t}(t,x) +\lambda\mbf{u}(t,x),	& &x\in(0,1),	\\
	\mbf{u}_{x}(t,0)&=\mbf{u}_{x}(t,1)=0,	&	&t\geq 0.
\end{align*}
We can solve this equation using separation of variables. Specifically, let $\mbf{u}(t,x)=T(t)X(x)$. Then, the functions $T$ and $X$ must satisfy
\begin{equation*}
	T''(t)X(x)=-2kT'(t)X(x) +T(t)[X''(x) +\lambda X(x)],
\end{equation*}
and therefore
\begin{equation*}
	\frac{T''(t)+2kT'(t)}{T(t)}-\lambda =\frac{X''(x)}{X(x)}.
\end{equation*}
It follows that both sides of this equation must be equal to some constant $-c\in\R$, so that
\begin{equation*}
	T''(t)=[\lambda-c]T(t) -2kT'(t),
	\qquad
	X''(x)=-cX(x).
\end{equation*}
Invoking the Neumann boundary conditions, we find that the ODE $X''(x)=-cX(x)$ admits non-trivial solutions only for $c=c_{n}=n^2\pi^2$ for $n\in\N_{0}$.
We obtain corresponding solutions to the ODE $T''(t)=[\lambda-c]T(t) +2kT'(t)$ as
\begin{align*}
	T_{n}(t)&=\gamma_{n} e^{(-k-\sqrt{(\lambda-c_{n})+k^2})t} +\delta_{n} e^{(-k+\sqrt{(\lambda-c_{n})+k^2})t}	\\
	&=e^{-kt}\bbl[\gamma_{n} e^{-t\sqrt{(\lambda-n^2\pi^2)+k^2}} +\delta_{n} e^{t\sqrt{(\lambda-n^2\pi^2)+k^2}}]
\end{align*}
We find that $T_{n}(t)$ is non-increasing only if $k\geq 0$ and $(\lambda-n^2\pi^2)\leq 0$. Setting $\lambda=-k^2$, the solutions become
\begin{equation*}
	T_{n}(t)=e^{-kt}\bbl[\gamma_{n} e^{-in\pi t} +\delta_{n} e^{in\pi t}\bbr].
\end{equation*}
It follows that the solutions to the original PDE will decay exponentially to zero only for $k>0$, with rate $\alpha=k$.

\end{appendices}

\end{document}